\documentclass[a4paper, 11pt, reqno]{amsart}
\usepackage{amsmath,amsfonts,amssymb,amsthm,enumerate}
\usepackage{graphicx}
 \usepackage{xy} \xyoption{all}

\newtheorem{thm}{Theorem}[section]
\newtheorem{cor}[thm]{Corollary}
\newtheorem{prop}[thm]{Proposition}
\newtheorem{lem}[thm]{Lemma}

\theoremstyle{definition}
\newtheorem{defn}[thm]{Definition}

\newtheorem{exas}[thm]{Examples}
\newtheorem{rem}[thm]{Remark}

\DeclareMathOperator{\Gr}{Gr}
\DeclareMathOperator{\gr}{gr}
\DeclareMathOperator{\QGr}{QGr}
\DeclareMathOperator{\qgr}{qgr}
\DeclareMathOperator{\proj}{proj}
\DeclareMathOperator{\Tors}{Tors}

\let\phi\varphi

\begin{document}
	\title{Automorphisms of Leavitt path algebras:  Zhang twist and irreducible representations}
	\maketitle
	\begin{center}
		Tran Giang Nam\footnote{Institute of Mathematics, VAST, 18 Hoang Quoc Viet, Cau Giay, Hanoi, Vietnam. E-mail address: \texttt{tgnam@math.ac.vn}},
		Ashish K. Srivastava\footnote{Department of Mathematics and Statistics, Saint  Louis University,  Saint Louis, MO-63103, USA. E-mail address: \texttt{ashish.srivastava@slu.edu}} and Nguyen Thi Vien\footnote{Institute of Mathematics, VAST, 18 Hoang Quoc Viet, Cau Giay, Hanoi, Vietnam. E-mail address: \texttt{nguyenthivien2000@gmail.com}
			
			\ \ {\bf Acknowledgements}: 
			The authors take an opportunity to express their deep gratitude to the anonymous referee for extremely careful reading, highly professional working with our manuscript, and valuable suggestions. The first author was supported by the Vietnam Institute for Advanced Study in Mathematics (VIASM) and by  the Vietnam Academy
of Science and Technology under grant CTTH00.01/24-25.		}
	\medskip
	
	Dedicated to Prof. Ngo Viet Trung on the occasion of his 70th birthday	
	\end{center}
	
	\begin{abstract} In this article, we construct (graded) automorphisms fixing all vertices of Leavitt path algebras of arbitrary graphs in terms of general linear groups over corners of these algebras. As an application, we study Zhang twist of Leavitt path algebras and describe new classes of irreducible representations of Leavitt path algebras of rose graphs $R_n$ with $n$ petals.  
		\medskip
	
	
\end{abstract}
\medskip

\section{Introduction}
\noindent The study of automorphisms of an algebra has been an important area of research as it describes the symmetry of underlying algebraic structure. But, determining the full automorphism group of a noncommutative algebra is, in general, an extremely difficult problem with very little progress till date. In 1968, Dixmier \cite{Dixmier} described the group of automorphisms of the first Weyl algebra. For higher Weyl algebras, to find the group of automorphisms is a long standing open problem. In \cite{Bavula}, Bavula described the group of automorphisms for the Jacobson algebra $\mathbb A_n=K\langle x, y \rangle/(xy-1)$ as a semidirect product of the multiplicative group $K^*$ of the field $K$ with the general linear finitary group $GL_{\infty}(K)$ using some deep arguments. The same result was later obtained by Alahmedi, Alsulami, Jain and Zelmanov in a remarkable work \cite{aajz} where they approach this problem from another perspective noting that the Jacobson algebra $\mathbb A_n$ is isomorphic to the Leavitt path algebra of Toeplitz graph and then they describe the group of automorphisms of this Leavitt path algebra. 

Leavitt path algebras were introduced independently by Abrams and Aranda Pino in \cite{ap:tlpaoag05} and Ara, Moreno and Pardo in \cite{amp:nktfga}. These are certain quotients of path algebras where the relations are inspired from Cuntz-Krieger relations for graph $C^*$-algebras (see \cite{kpr:ckaodg, r:ga}). For a graph $E$ that has only one vertex and $n$ loops, Leavitt path algebra turns out to be the algebra of type $(1, n)$ proposed by Leavitt as an example of a (universal) ring without invariant basis number (see \cite{leav:tmtoar}).  Leavitt path algebras have deep connections with symbolic dynamics and the theory of graph $C^*$-algebras. For example, the notion of flow equivalence of shifts of finite type in symbolic dynamics is related to Morita theory and the Grothendieck group in the theory of Leavitt path algebras, and ring isomorphism (or Morita equivalence) between two Leavitt path algebras over the field of the complex numbers induces, for some graphs, isomorphism (or Morita equivalence) of the respective graph $C^*$-algebras. As remarked by Chen in \cite{c:irolpa}, Leavitt path algebras capture the homological properties of both path algebras and their Koszul dual and hence they form an important class of noncommutative algebras. Moreover, by Smith's interesting result (\cite[Theorem 1.3]{Smith2012}), the Leavitt path algebra construction arises naturally in the context of noncommutative algebraic geometry. We refer the reader to \cite{a:lpatfd, AAS:LPA} for a detailed history and overview of these algebras. 

Unfortunately, there are not many constructions known yet for automorphisms of Leavitt path algebras. In \cite{ajs:vaoga, jss:tpeoga}, motivated by Cuntz's idea \cite{cuntz:aocaCa},
Szyma\'{n}ski et al. gave a method to construct  automorphisms  of Leavitt path algebras $L_K(E)$ of finite graphs $E$ without sinks or sources in which every cycle has an exit over integral domains $K$ of characteristic $0$. In \cite[Section 2]{kn:ataanirolpa}, Kuroda and the first author  gave construction of  automorphisms fixing all vertices of Leavitt path algebras $L_K(E)$ of arbitrary graphs $E$ over an arbitrary field $K$, and gave construction of Anick type automorphisms of Leavitt path algebras. Anick automorphisms have an interesting history. For a free associative algebra $F<x, y, z>$ over a field $F$ of characteristic zero, the question about the existence of a wild automorphism was open for a long time. Anick provided a candidate for a wild automorphism in the case of free associative algebra on three generators. In \cite{Umirbaev}, Umirbaev proved that the Anick automorphism $\delta=(x+z(xz-zy), y+(xz-zy)z, z)$ of the algebra $F\langle x, y, z\rangle$ over a field $F$ of characteristic zero is wild. In this paper, based on Kuroda and the first author's work \cite[Section 2]{kn:ataanirolpa} and Cuntz's beautiful paper \cite{cuntz:aocaCa}, we give a construction for graded automorphisms of Leavitt path algebras. We describe (graded) automorphisms fixing all vertices of Leavitt path algebras of arbitrary graphs in terms of general linear groups over corners of these algebras (Theorem~\ref{Iso} and Corollary \ref{gr-Iso}). Consequently, this yields a complete description of all (graded) automorphisms of the Leavitt path algebra $L_K(R_n)$ of the rose graph $R_n$ with $n$ petals in term of general linear group of degree $n$ over $L_K(R_n)$ (Propositions \ref{IsoofLA} and \ref{gr-IsoofLA}). Moreover, we show that the group of all graded automorphisms of $L_K(R_n)$ contains some special subgroups, for example, the general linear group of degree $n$ over $K$ (Corollaries \ref{gl-IsoofLA} and \ref{Anick-IsoofLA}). We also provide a complete description of all (graded) automorphisms of $L_K(R_n)$ via the group of units $U(L_K(R_n))$ of $L_K(R_n)$ (Corollary \ref{Cuntz-IsoofLA}).

As the first application of these constructions for graded automorphisms, we study twists of Leavitt path algebras. One of the most frequently used tools to construct new examples of algebras and coalgebras is twisting the multiplicative structure of original algebra. Classic examples of algebras constructed by twisting multiplicative structure include skew polynomial rings and skew group rings. The twist of Leavitt path algebras that we study here is a twist in the sense of Artin, Tate and Van den Bergh. A notion of twist of a graded algebra $A$ was introduced by Artin, Tate, and Van den Bergh in \cite{atv:moraod3} as a deformation of the original graded product of $A$ with the help of a graded automorphism of $A$. Let $\sigma$ be an automorphism of the graded algebra $A=\oplus A_n$. Define a new multiplication $\star$ on the underlying graded $K$-module $\oplus A_n$ by $a \star b=a\sigma^n(b)$ where $a$ and $b$ are homogeneous elements in $A=\oplus A_n$ and $deg(a)=n$. The new graded algebra with the same underlying graded $K$-module $\oplus A_n$ and the new graded product $\star$ is called the twist of $A$ and is denoted as $A^{\sigma}$. 

This notion of twist of a graded algebra was later generalized by Zhang in \cite{zhang:tgaaeogc}, where he introduced the concept of twisting of graded product with the help of a twisting system. Let $\tau=\{\tau_n\mid n \in \mathbb Z\}$ be a set of graded $K$-linear automorphisms of $A=\oplus A_n$. Then $\tau$ is called a twisting system if $\tau_n(y\tau_m(z))=\tau_n(y)\tau_{n+m}(z)$ for all $n, m, l\in \mathbb Z$ and $y\in A_m$, $z\in A_l$. For example, if $\sigma$ is a graded algebra automorphism of $A$, then $\tau=\{\sigma^n\mid n\in \mathbb Z\}$ is a twisting system. Thus, the twist of a graded algebra in the sense of Artin-Tate-Van den Bergh can be viewed as a special case of the twist introduced by Zhang. Such a twist of a graded algebra is now known as Zhang twist. 

Zhang twist of a graded algebra has played a vital role in the interaction of noncommutative algebra with noncommutative projective geometry. The fundamental idea behind the noncommutative projective scheme defined by Artin and Zhang \cite{AZ} is to give up on the actual geometric space and instead generalize only the category of coherent sheaves to the noncommutative case. In the case of commutative algebras, Serre's theorem established that studying the category of quasi-coherent sheaves on a projective variety is essentially the same as studying the quotient category of graded modules. The definition of noncommutative projective space is motivated by Serre's result. 

Let $A$ be a right noetherian graded algebra. We denote by $\Gr-A$ the category of graded right $A$-modules with morphisms being graded homomorphisms of degree zero. An element $x$ of a graded right $A$-module $M$ is called torsion if $xA_{\ge s}=0$ for some $s$. The torsion elements in $M$ form a graded $A$-submodule which is called the torsion submodule of $M$. The torsion modules form a subcategory for which we use the notation $\Tors(A):=$ the full subcategory of $\Gr-A$ of torsion modules. We denote $\QGr-A:=$ the quotient category $\Gr-A/\Tors(A)$. We will use the lower case notations $\gr-A$, $\qgr-A$ to indicate that we are working with finitely generated $A$-modules. Since $\qgr-A$ is a quotient category of $\gr-A$, it inherits two structures: the object $\mathcal A$ which is the image in $\qgr-A$ of $A_A$, and the shift operator $s$ on $\qgr-A$, which is the automorphism of the category $\qgr-A$ determined by the shift on $\gr-A$. The triple $(\qgr-A, \mathcal A, s)$ is called the noncommutative projective scheme associated to $A$, denoted as $\proj-A$. We refer the reader to \cite{AZ} for more details on noncommutative projective scheme. 

One of the main features of the study of Zhang twist of a graded algebra is that if an algebra $B$ is isomorphic to the Zhang twist of an algebra $A$, then their graded module categories $\Gr-A$ and $\Gr-B$ are equivalent. If the algebra $A$ is noetherian, then this equivalence restricts to the subcategories of finitely generated modules to give an equivalence $\gr-A \cong \gr-B$. Moreover, the subcategories of modules which are torsion (that is, finite-dimensional over $K$) also correspond, and so we have an equivalence between the quotient categories $\qgr-A$ and $\qgr-B$. As a consequence it follows that their noncommutative projective schemes $\proj-A$ and $\proj-B$ are equivalent. Since Zhang twist of a commutative graded algebra by a non-identity automorphism yields a noncommutative graded algebra, this gives us a tool to construct examples of noncommutative graded algebras whose noncommutative projective schemes are isomorphic to commutative projective schemes. It is known that many fundamental properties like Gelfand-Kirillov dimesnion and Artin-Schelter regularity are preserved under Zhang twist whereas some ring-theoretic properties such as being a prime ring or being a PI ring are not preserved under Zhang twist. 

In this paper we initiate the study of Zhang twist of Leavitt path algebras with a larger goal to develop the geometric theory of Leavitt path algebras. In Section 3, we twist the multiplicative structure of Leavitt path algebras with the help of graded automorphisms constructed in Section 2. In a rather surprising result we show that the Leavitt path algebra $L_K(E)$ of an arbitrary graph $E$ may be embedded into the Zhang twist $L_K(E)^{\phi_P}$ by any graded automorphism $\phi_P$ introduced in Corollary \ref{gr-Iso} (Proposition \ref{ZhangTw-SubalLPA}), and the embedding is not an isomorphism in general (Examples \ref{exa-theta} (3)). Geometrically, this means that any noetherian Leavitt path algebra always embeds in another algebra with the same projective scheme. We also characterize Leavitt path algebras $L_K(R_n)$ of the rose graph $R_n$ with $n$ petals that are rigid to Zhang twist in the sense that $L_K(R_n)$ turns out to be isomorphic to its Zhang twist with respect to graded automorphisms constructed in Section 2 (Theorem \ref{ZhTwisotoLeAl} and Corollary \ref{theta-u}).   

Automorphism of an algebra helps in constructing new twisted irreducible representations. It is not difficult to see that if $M$ is an irreducible representation of an algebra $A$ and $\sigma$ is an automorphism of $A$ then $M^{\sigma}$ is also an irreducible representation where $M^{\sigma}$ is the same vector space as $M$ with the module operation given as $a.m=\sigma(a)m$ for any $a\in A$. This new irreducible representation $M^{\sigma}$ of $A$ is called a twisted representation. In another application to our constructions of automorphisms, we study the irreducible representations of the Leavitt path algebra of rose graph $R_n$ with $n$ petals in the last section of this paper. 

In a seminal work \cite{c:irolpa}, Chen constructed irreducible representations of Leavitt path algebras using infinite paths. For an infinite path $p$ in $E$, Chen constructed a simple module $V_{[p]}$ for the Leavitt path algebra $L_K(E)$ of an arbitrary graph $E$ where $[p]$ is the equivalence class of infinite paths tail-equivalent to $p$. Later, in \cite{ar:fpsmolpa}, Ara and Rangaswamy characterized Leavitt path algebras which admit only finitely presented irreducible representations. In \cite{anhnam}, \'{A}nh and the first author constructed a new class of simple $L_K(E)$-modules, $S^f_c$ associated to pairs $(f, c)$ consisting of simple closed paths $c$ together with irreducible polynomials $f\in K[x]$.  We should note that Ara and Rangaswamy \cite{ar:fpsmolpa} classified all simple modules over the Leavitt path algebra of a finite graph in which every vertex is in at most one cycle. Their result induces our investigation of simple modules for Leavitt path algebras of graphs having a vertex that is in at least two cycles.
The most important case of this class is the Leavitt path algebra of a rose graph with $n\ge 2$ petals.   

For Leavitt path algebra $L_K(R_n)$ of the rose graph $R_n$ with $n$ petals, in \cite{kn:ataanirolpa} Kuroda and the first author constructed additional classes of simple $L_K(R_n)$-modules by studying the twisted modules of the simple modules $S^f_c$ under Anick type automorphisms of $L_K(R_n)$ mentioned in Corollary \ref{Anick-IsoofLA}. In Section 4, we define a new simple left $L_K(R_n)$-module $V_{[\alpha]}^P:=(V_{[\alpha]})^{\phi^{-1}_P}$ which is a twist of the simple $L_K(R_n)$-module $V_{[\alpha]}$ by the graded automorphism $\phi^{-1}_P$ mentioned in Proposition \ref{gr-IsoofLA}, where $\alpha$ is an infinite path in $R_n$ and $P\in GL_n(K)$, and classify completely these simple modules (Theorems \ref{Irrrep2-irr} and \ref{Irrrep5-irr}). Also, in Theorem \ref{Irrrep2-irr}, we show that $V^P_{[\alpha]}\cong L_K(R_n)/\bigoplus^{\infty}_{m=0}L_K(R_n)(\phi_P(\epsilon_m) - \phi_P(\epsilon_{m+1}))$ for all irrational path $\alpha = e_{i_1}\cdots e_{i_m}\cdots$, where $\epsilon_0 := v$, $\epsilon_m = e_{i_1}\cdots e_{i_m}e^*_{i_m}\cdots e^*_{i_1}$ for all $m\ge 1$, and the graded automorphism $\phi_P$ is defined in Proposition \ref{gr-IsoofLA}. Consequently, $V^P_{[\alpha]}$ is not finitely presented. Moreover, we show that there are infinitely many isomorphic classes of these simple modules (Corollaries \ref{Irrrep3-irr} and \ref{Irrrep4-irr}). For a simple closed path $c$ in $R_n$, we show in Theorem \ref{Irrrep5-irr} that the twisted module $V^P_{[c^{\infty}]}$ is finitely presented for all $P\in GL_n(K)$ and $V^P_{[c^{\infty}]}\cong L_K(R_n)/L_K(R_n)(v - \phi_P(c))$. Furthermore, we obtain that there are infinitely many isomorphic classes of these simple modules (Corollary \ref{Irrrep6-irr}). 
We conclude this paper by giving a list of some classes of pairwise non-isomorphic simple modules over $L_K(R_n)$ in Theorem \ref{Irrrep7-irr}. 

Throughout the paper, all {\it algebras} are associative over a field, not necessary with identity but having local units. {\it Modules} are considered with respect to algebras, that means, they are also at the same time vector spaces over a field. Modules are {\it unitary} in the sense that each of their elements is fixed by an appropriate idempotent. {\it Algebra homomorphisms} are defined in the standard manner. Homomorphisms between algebras with identity are algebra homomorphisms having the identity-preserving property.

\section{On graded automorphisms of Leavitt path algebras}
\noindent Cuntz \cite{cuntz:aocaCa} showed that there is a one-to-one correspondence between unitary elements of the Cuntz algebra $\mathcal{O}_n$ and endomorphisms of $\mathcal{O}_n$  via $u\longmapsto \lambda_u$ where $\lambda_u (S_i) = uS_i$, and provided criteria for these endomorphisms to be automorphisms. In \cite{chs:eoga}, motivated by Cuntz's results, Conti, Hong and Szyma\'{n}ski introduced a class of endomorphisms fixing all vertex projections $\lambda_u$ of $C^*(E)$ corresponding to unitaries in the multiplier algebra $M(C^*(E))$ which commute with all vertex projections. Then, they studied localized automorphisms of the graph algebra $C^*(E)$ of a finite graph without sink (i.e., automorphisms $\lambda_u$ corresponding to unitaries $u$ from the algebraic part of the core AF-subalgebra which
commute with the vertex projections), and gave combinatorial criteria for localized endomorphisms corresponding to permutation unitaries to be automorphisms. 

Szyma\'{n}ski et al. \cite{ajs:vaoga, jss:tpeoga} studied permutative automorphisms and  polynomial endomorphisms of graph $C^*$-algebras $C^*(E)$ and Leavitt path algebras $L_K(E)$, where $E$ is a finite graph without sinks or sources in which every cycle has an exit, and $K$ is an integral domain of characteristic $0$. Kuroda and the first author \cite[Section 2]{kn:ataanirolpa} gave a method to construct endomorphisms and automorphisms fixing all vertices of Leavitt path algebras $L_K(E)$ of arbitrary graphs $E$ over an arbitrary field $K$, by using special pairs $(P, Q)$ consisting of matrices in $M_n(L_K(E))$ which commute with all vertices in $E$, where $n$ is an arbitrary positive integer.

The first aim of this section is to completely describe endomorphisms introduced in \cite{kn:ataanirolpa}, and give criteria for these endomorphisms to be automorphisms. Before giving these constructions for automorphisms of Leavitt path algebras, we begin this section by recalling some useful notions of graph theory. 

A (\textit{directed}) \textit{graph} is a quadruplet $E = (E^0, E^1, s, r)$ which consists of two disjoint sets $E^0$ and $E^1$, called the set of \emph{vertices} and the set of \emph{edges}
respectively, together with two maps $s, r: E^1 \longrightarrow E^0$.  The vertices $s(e)$ and $r(e)$ are referred to as the \emph{source} and the \emph{range} of the edge~$e$, respectively. 
A vertex~$v$ for which $s^{-1}(v)$ is empty is called a \emph{sink}; a vertex~$v$ is \emph{regular} if $0< |s^{-1}(v)| < \infty$;  a vertex~$v$ is an \textit{infinite emitter} if $|s^{-1}(v)| = \infty$; and a vertex is \textit{singular} if it is either a sink or an infinite emitter. 

A \emph{finite path of length} $n$ in a graph $E$ is a sequence $p = e_{1} \cdots e_{n}$  of edges $e_{1}, \dots, e_{n}$ such that $r(e_{i}) = s(e_{i+1})$ for $i = 1, \dots, n-1$.  In this case, we say that the path~$p$ starts at the vertex $s(p) := s(e_{1})$ and ends at the vertex $r(p) := r(e_{n})$, we write $|p| = n$ for the length of $p$.  We consider the elements of $E^0$ to be paths of length $0$. We denote by $E^*$ the set of all finite paths in $E$.  An edge $f$ is an \textit{exit} for a path $p= e_1 \cdots e_n$ if $s(f) = s(e_i)$ but $f \neq e_i$ for some $1\le i\le n$. A finite path $p$ of positive length is called a \textit{closed path based at} $v$ if $v = s(p) = r(p)$. A \textit{cycle} is a closed path $p = e_{1} \cdots e_{n}$, and for which the vertices $s(e_1), s(e_2), \hdots, s(e_n)$ are distinct. A closed path $c$ in $E$ is called \textit{simple} if $c \neq d^n$ for any closed path $d$ and integer $n\ge 2$. We denoted by $SCP(E)$ the set of all simple closed paths in $E$.


\begin{defn}\label{LPA}
	For an arbitrary graph $E = (E^0,E^1,s,r)$
	and any  field $K$, the \emph{Leavitt path algebra} $L_{K}(E)$ {\it of the graph}~$E$ \emph{with coefficients in}~$K$ is the $K$-algebra generated by the union of the set $E^0$ and two disjoint copies of $E^1$, say $E^1$ and $\{e^*\mid e\in E^1\}$, satisfying the following relations for all $v, w\in E^0$ and $e, f\in E^1$:
	\begin{itemize}
		\item[(1)] $v w = \delta_{v, w} w$;
		\item[(2)] $s(e) e = e = e r(e)$ and $e^*s(e)  = e^* = r(e) e^*$;
		\item[(3)] $e^* f = \delta_{e, f} r(e)$;
		\item[(4)] $v= \sum_{e\in s^{-1}(v)}ee^*$ for any regular vertex $v$;
	\end{itemize}
	where $\delta$ is the Kronecker delta.
\end{defn}
If $E^0$ is finite, then $L_K(E)$ is a unital ring having identity $1 = \sum_{v\in E^0}v$ (see, e.g. \cite[Lemma 1.6]{ap:tlpaoag05}).
It is easy to see that the mapping given by $v\longmapsto v$ for all $v\in E^0$, and $e\longmapsto e^*$, $e^*\longmapsto e$ for all $e\in E^1$, produces an involution on the algebra $L_K(E)$, and for any path $p = e_1e_2\cdots e_n$, the element $e^*_n\cdots e^*_2e^*_1$ of $L_K(E)$ is denoted by $p^*$.
It can be shown (\cite[Lemma 1.7]{ap:tlpaoag05}) that $L_K(E)$ is  spanned as a $K$-vector space by $\{pq^* \mid p, q\in E^*, r(p) = r(q)\}$. Indeed, $L_K(E)$ is a $\mathbb{Z}$-graded $K$-algebra: $L_K(E) = \oplus_{n\in \mathbb{Z}}L_K(E)_n$, where for each $n\in \mathbb{Z}$, the degree $n$ component $L_K(E)_n$ is the set $ \text{span}_K \{pq^*\mid p, q\in E^*, r(p) = r(q), |p|- |q| = n\}$.

The Leavitt path algebra $L_K(E)$ of a graph $E$ over field $K$ has the following universal property: if $\mathcal{A}$ is a $K$-algebra generated by a family of elements $\{a_v, b_e, c_{e^*}\mid v\in E^0, e\in E^1\}$ satisfying the relations analogous to (1) - (4)  in Definition~\ref{LPA}, then there exists a unique $K$-algebra homomorphism $\varphi: L_K(E)\longrightarrow \mathcal{A}$ given by $\varphi(v) = a_v$, $\varphi(e) = b_e$ and $\varphi(e^*) = c_{e^*}$.  We will refer to this property as the Universal Property of $L_K(E)$.

As usual, for any ring $R$, for any endomorphism $f\in End(R)$ and for any $A\in M_n(R)$, we denote by $f(A)$ the matrix $(f(a_{i,j}))\in M_n(R)$, and denote by $A_m$ the matrix $Af(A)\cdots f^{m-1}(A)\in M_n(R)$ for every $m\ge 1$, where $f^0:=id_R$. For any $\mathbb{Z}$-graded algebra $A$ over a field $K$, we denote by $End^{gr}(A)$ the $K$-algebra of all graded endomorphisms of  $A$, and denote by $Aut^{gr}(A)$ the group of all graded automorphisms of $A$.

We are now in a position to establish the main result of this section providing a method to construct endomorphisms and automorphisms fixing all vertices of Leavitt path algebras of arbitrary graphs over an arbitrary field in terms of general linear groups over corners of these algebras.

\begin{thm}\label{Iso}
	Let $K$ be a field, $n$ a positive integer, $E$ a graph, and $v$ and $w$ vertices in $E$ (they may be the same). Let $e_1, e_2, \ldots, e_n$ be distinct edges in $E$ with $s(e_i) = v$ and $r(e_i) = w$ for all $1\le i\le n$.  Let $P$ be an element of $GL_n(wL_K(E)w)$ with $P=(p_{i,j})$ and $P^{-1}=(p'_{i,j})$. Then the following statements hold:
	
	$(1)$ There exists a unique injective homomorphism $\phi_{P}: L_K(E)\longrightarrow L_K(E)$ of $K$-algebras satisfying 
	\begin{center}
		$\phi_{P}(u)=u,\quad \phi_{P}(e)=e\quad\text{ and }\quad \phi_{P}(e^*)=e^*$ \end{center} for all $u\in E^0$ and $e\in E^1\setminus \{ e_1,\ldots ,e_n\} $, and 
	\begin{center}
		$\phi_{P}(e_i) = \sum^n_{k=1}e_kp_{k,i}$\quad	and \quad $\phi_{P}(e^*_i) = \sum^n_{k=1}p'_{i,k}e^*_k$
	\end{center} for all $1\le i\le n$. 
	
	$(2)$ For every $Q\in GL_n(wL_K(E)w)$, $\phi_P = \phi_Q$ if and only if $P=Q$. Consequently, $\phi_P = id_{L_K(E)}$ if and only if $P$ is the identity matrix of $M_n(wL_K(E)w)$.
	
	$(3)$ $\phi_P\phi_Q = \phi_{_P\phi_P(Q)}$ for all $Q\in GL_n(wL_K(E)w)$. In particular, $\phi^m_P = \phi_{P_m}$ for all positive integer $m$. 
	
	$(4)$ $\phi_P$ is an isomorphism if and only if $P^{-1} = \phi_P(Q)$ for some $Q\in GL_n(wL_K(E)w)$. In this case, $\phi_{P_m}^{-1} = \phi_{Q_m}$, $P_m = \phi_{P_m}(Q^{-1}_m)$ and $P^{-1}_m = \phi_{P_m}(Q_m)$ for all $m\ge 1$. In particular, if $\phi_{P}(P)= P$ or $\phi_{P}(P^{-1})= P^{-1}$, then $\phi_{P}$ is an isomorphism and $\phi_{P}^{m}=\phi_{P^{m}}$ for all integer $m$.
	
	If, in addition, $|s^{-1}(v)| = n$, then we have the following:
	
	$(5)$ For every $K$-algebra homomorphism $\lambda: L_K(E)\longrightarrow L_K(E)$ with $\lambda(u)=u,\ \lambda(e)=e \text{ and } \lambda(e^*)=e^*$  for all $u\in E^0$ and $e\in E^1\setminus\{ e_1,\ldots ,e_n\}$, there exists a unique matrix $P = (p_{i,j})\in GL_n(wL_K(E)w)$ such that $p_{i, j} = e^*_i\lambda(e_j)$ for all $1\le i, j\le n$ and $\lambda = \phi_P$.
	
	$(6)$ We denote by $End_{v,w}(L_K(E))$ the set of all endomorphisms $\lambda$ of $L_K(E)$ with  $\lambda(u)=u,\ \lambda(e)=e \text{ and } \lambda(e^*)=e^*$  for all $u\in E^0$ and $e\in E^1\setminus \{ e_1,\ldots ,e_n\}$. Then, the map $\Phi: (GL_n(wL_K(E)w), \star)\longrightarrow End_{v,w}(L_K(E))$, $P\longmapsto \phi_P$, is a monoid isomorphism, where the multiplication law ``$\star$" is defined by $$P\star Q = P\phi_P(Q)$$ for all $P, Q\in GL_n(wL_K(E)w)$.
\end{thm}
\begin{proof} 
	(1) The existence of a unique homomorphism $\phi_{P}: L_K(E)\longrightarrow L_K(E)$ of $K$-algebras with the desired property follows from \cite[Theorem 2.2 (i)]{kn:ataanirolpa}. For the sake of completeness, we give a sketch of the proof. We define the elements $\{Q_u: u\in E^0\}$ and $\{T_e, T_{e^*}: e\in E^1\}$ by setting $Q_u=u$, 
	\begin{equation*}
	T_{e}=  \left\{
	\begin{array}{lcl}
	\sum^n_{k=1} e_kp_{k,i}&  & \text{if } e= e_i \text{ for some } 1\le i\le n\\
	e  & \text{otherwise}.
	\end{array}%
	\right.
	\end{equation*}%
	and
	\begin{equation*}
	T_{e^*}=  \left\{
	\begin{array}{lcl}
	\sum^n_{k=1}p'_{i,k}e^*_k&  & \text{if } e= e_i \text{ for some } 1\le i\le n\\
	e^*&  & \text{otherwise}.
	\end{array}%
	\right.
	\end{equation*}%
	and show that these elements form a generating set for $L_K(E)$ with the same relations as defining relations for Leavitt path algebra (see \cite[Theorem 2.2 (i)]{kn:ataanirolpa} for details). Therefore, by the Universal Property of Leavitt path algebras, there exists a unique homomorphism 
	$\phi_{P}: L_K(E)\longrightarrow L_K(E)$ of $K$-algebras satisfying $\phi_{P}(u)=Q_u$, $\phi_{P}(e)=T_e$, $\phi_{P}(e^*)=T_{e^*}$ for all $u\in E^0$ and $e\in E^1$. Consequently, we have 
	$\phi_{P}(u)=u, \phi_{P}(e)=e\text{ and } \phi_{P}(e^*)=e^*$  for all $u\in E^0$ and $e\in E^1\setminus \{ e_1,\ldots ,e_n\} $, and 
	\begin{center}
		$\phi_{P}(e_i) = \sum^n_{k=1}e_kp_{k,i}$\quad	and \quad $\phi_{P}(e^*_i) = \sum^n_{k=1}p'_{i,k}e^*_k$
	\end{center} for all $1\le i\le n$.

	We next prove that $\phi_{P}$ is injective by following the proof of \cite[Theorem 2.2 (ii)]{kn:ataanirolpa}. To the contrary, suppose there exists a nonzero element $x\in\ker(\phi_{P})$. Then, by the Reduction Theorem (see, e.g., \cite[Theorem 2.2.11]{AAS:LPA}), there exist $a, b \in L_K(E)$ such that either $axb = u \neq 0$ for some $u\in E^0$, or $axb = p(c)\neq 0$, where $c$ is a cycle in $E$ without exits and $p(x)$ is a nonzero polynomial in $K[x, x^{-1}]$.
	
	In the first case, since $axb\in \ker(\phi_{P})$, this would imply that $u= \phi_{P}(u) = 0$ in $L_K(E)$; but each vertex is well-known to be a nonzero element inside the Leavitt path algebra, which is a contradiction.
	
	So we are in the second case: there exists a cycle $c$ in $E$ without exits such that $axb = \sum^m_{i=-l}k_ic^i\neq 0$, where $k_i\in K$, $l$ and $m$ are nonnegative integers, and we interpret $c^i$ as $(c^*)^{-i}$ for negative $i$, and we interpret $c^0$ as $u := s(c)$. Write $c = g_1g_2\cdots g_t$, where $g_i\in E^1$ and $t$ is a positive integer. If $g_i\in E^1\setminus \{ e_1,\ldots ,e_n\}$ for all $1\le i\le t$, then $\phi_{P}(c) = c$ and $\phi_{P}(c^*) = c^*$, so $0\neq \sum^m_{i=-l}k_ic^i = \sum^m_{i=-l}k_i\phi_{P}(c^i) = \phi_{P}(axb) = 0$ in $L_K(E)$, a contradiction. Consider the case that there exists a $1\le k\le t$ such that $g_k = e_i$ for some $i$. Then, since $c$ is a cycle without exits, we must have $n = 1$ and $k$ is a unique element such that $g_k = e_1$. Let $\alpha := g_{k+1}\cdots g_tg_1\cdots g_{k-1}e_1$. We have that $\alpha$ is a cycle in $E$ without exits and $s(\alpha) = w$. Since $n=1$, $P= p_{1, 1}$ and $P^{-1}= p'_{1, 1}$ are two elements of $wL_K(E)w$ with $p_{1, 1}p'_{1, 1} = w= p'_{1, 1}p_{1, 1}$, so $p_{1,1}$ is a unit of $wL_K(E)w$ with $p_{1,1}^{-1} = p'_{1,1}$. 
	By \cite[Lemma 2.2.7]{AAS:LPA}, we have \[wL_K(E)w = \{\sum^h_{i=l}k_i\alpha^i\mid k_i\in K, l\le h, h, l\in \mathbb{Z}\}\cong K[x, x^{-1}]\]via an isomorphism that sends $v$ to $1$, $\alpha$ to $x$ and $\alpha^*$ to $x^{-1}$, and so $p_{1,1} = a\alpha^s$ and $p'_{1,1} = a^{-1}\alpha^{-s}$ for some $a\in K\setminus \{0\}$ and $s\in \mathbb{Z}$. If $s\ge 0$, then 
	
	\begin{align*}
	\phi_{P}(c) & =\phi_{P}(g_1\cdots g_{k-1}e_1g_{k+1}\cdots g_t) =(g_1\cdots g_{k-1})e_1p_{1,1} (g_{k+1}\cdots g_t)=\\
	&=(g_1\cdots g_{k-1}e_1)a\alpha^s (g_{k+1}\cdots g_t)= a(g_1\cdots g_{k-1}e_1)\alpha^s (g_{k+1}\cdots g_t)=
	ac^{s+1},
	\end{align*} 
	and 
	\begin{align*}
	\phi_{P}(c^*) & =\phi_{P}(g^*_{t}\cdots g^*_{k+1}e^*_1g^*_{k-1}\cdots g^*_{1}) =(g^*_{t}\cdots g^*_{k+1})p'_{1,1}e^*_1 (g^*_{k-1}\cdots g^*_{1})=\\
	&=a^{-1}(g^*_{t}\cdots g^*_{k+1})\alpha^{-s}(e^*_1 g^*_{k-1}\cdots g^*_{1})= (g^*_{t}\cdots g^*_{k+1})(\alpha^*)^{s}(e^*_1 g^*_{k-1}\cdots g^*_{1})=\\& = a^{-1}(c^*)^{s+1}.
	\end{align*} 
	If $s< 0$, then 
	
	\begin{align*}
	\phi_{P}(c) & =\phi_{P}(g_1\cdots g_{k-1}e_1g_{k+1}\cdots g_t) =(g_1\cdots g_{k-1})e_1p_{1,1} (g_{k+1}\cdots g_t)=\\
	&=(g_1\cdots g_{k-1}e_1)a\alpha^s (g_{k+1}\cdots g_t)= a(g_1\cdots g_{k-1}e_1)(\alpha^*)^{-s} (g_{k+1}\cdots g_t)=\\
	&=a (c^*)^{-s-1} = a(c^*)^{-s-1} = ac^{s+1},
	\end{align*} 
	and 
	\begin{align*}
	\phi_{P}(c^*) & =\phi_{P}(g^*_{t}\cdots g^*_{k+1}e^*_1g^*_{k-1}\cdots g^*_{1}) =(g^*_{t}\cdots g^*_{k+1})p'_{1,1}e^*_1 (g^*_{k-1}\cdots g^*_{1})=\\
	&=a^{-1}(g^*_{t}\cdots g^*_{k+1})\alpha^{-s}(e^*_1 g^*_{k-1}\cdots g^*_{1})= (g^*_{t}\cdots g^*_{k+1})(\alpha^*)^{-s}(e^*_1 g^*_{k-1}\cdots g^*_{1})=\\& = a^{-1}(c^*)^{-s-1} =  a^{-1}c^{s+1}.
	\end{align*} 
	Therefore, we obtain that $\phi_{P}(c^l) = a^lc^{l(s+1)}$ for all $l\in \mathbb{Z}$, and $$0\neq \sum^m_{i=-l}k_ia^ic^{i(s+1)} = \sum^m_{i=-l}k_i\phi_{P}(c^i) = \phi_{P}(axb) = 0$$ in  $L_K(E)$, which is a contradiction.
	
	In any case, we arrive at a contradiction, and so we infer that $\phi_{P}$ is injective, as desired.
	
	(2) Assume that $Q= (q_{i,j})\in GL_n(wL_K(E)w)$ and $\phi_P = \phi_Q$. We then have $\sum^n_{k=1}e_kp_{k,j}=\phi_{P}(e_j) = \phi_Q(e_j)= \sum^n_{k=1}e_kq_{k,j}$ for all $1\le j\le n$, and so $$p_{i,j} = wp_{i,j} = e^*_i(\sum^n_{k=1}e_kp_{k,j})=e^*_i(\sum^n_{k=1}e_kq_{k,j}) = wq_{i,j}= q_{i,j}$$
	for all $1\le i, j\le n$. This implies that $P=Q$. The converse is obvious.
	
	(3) Suppose $Q$ is an element of $GL_n(wL_K(E)w)$ with $Q=(q_{i,j})$ and $Q^{-1}=(q'_{i,j})$. We then have  $P\phi_P(Q)\in GL_n(wL_K(E)w)$ and $(P\phi_P(Q))^{-1} = \phi_P(Q^{-1})P^{-1}$.
	
	We claim that $\phi_P\phi_Q = \phi_{_P\phi_P(Q)}$. It suffices to check that 
	\begin{center}
		$\phi_{P}\phi_{Q}(e_i) = \phi_{_P\phi_P(Q)}(e_i)$ and $\phi_{P}\phi_{Q}(e^*_i) = \phi_{_P\phi_P(Q)}(e^*_i)$ for all $1\le i\le n$.	
	\end{center}
	
	For each $1\le i\le n$, by definition of $\phi_{Q}$, $\phi_{Q}(e_i) = \sum^n_{k=1}e_kq_{k, i}$ and $\phi_{Q}(e^*_i) = \sum^n_{k=1}q'_{i, k}e^*_k$,  so 
	\begin{align*}
	\phi_{P}\phi_{Q}(e_i) &=\phi_{P}(\sum^n_{k=1}e_kq_{k, i})=\sum^n_{k=1}\phi_{P}(e_k)\phi_{P}(q_{k,i})=\sum^n_{k=1}\sum^n_{l=1}e_lp_{l,k}\phi_{P}(q_{k,i})\\
	&=\sum^n_{l=1}e_l(\sum^n_{k=1}p_{l,k}\phi_{P}(q_{k,i}))= \phi_{_P\phi_P(Q)}(e_i)
	\end{align*} 
	and 
	\begin{align*}
	\phi_{P}\phi_{Q}(e^*_i)&=\phi_{P}(\sum^n_{k=1}q'_{i, k}e^*_k)=\sum^n_{k=1}\phi_{P}(q'_{i, k})\phi_{P}(e^*_k)= \sum^n_{k=1}\sum^n_{l=1}\phi_{P}(q'_{i, k})p'_{k,l}e^*_l\\
	& =\sum^n_{l=1}(\sum^n_{k=1}\phi_{P}(q'_{i, k})p'_{k,l})e^*_l= \phi_{_P\phi_P(Q)}(e^*_i),
	\end{align*} 
	proving the claim. 
	
	We show that  $\phi^m_P = \phi_{P_m}$ for all positive integer $m$. First, note that $P_m\in GL_n(wL_K(E)w)$ with $P^{-1}_m = \varphi^{m-1}_P\left(P^{-1}\right)\cdots\varphi_P\left(P^{-1}\right)P^{-1}$. We use induction on $m$ to establish the fact $\phi^m_P = \phi_{P_m}$ for all $m\ge 1$. If $m=1$, then the fact is obvious. Now we proceed inductively. For $m >1$, by the induction hypothesis, $\phi^{m-1}_P = \phi_{P_{m-1}}$, and so 
	$$\phi_{P}^{m} = \phi_P \phi^{m-1}_P = \phi_P\phi_{P_{m-1}}= \phi_{_P\phi_P(P_{m-1})}=\phi_{P_m},$$ as desired.
	
	(4) ($\Rightarrow$) Assume that $\phi_P$ is an isomorphism, that means, there exists a matrix $Q\in GL_n(wL_K(E)w)$ such that $\phi_P\phi_Q= id_{L_K(E)}$. Then, by item (3), $\phi_{_P\phi_P(Q)} = id_{L_K(E)}$, and so $P\phi_P(Q)$ is the identity of $M_n(wL_K(E)w)$ by item (2). This shows that $P^{-1} = \phi_P(Q)$.
	
	$(\Leftarrow)$ Assume that $P^{-1} = \phi_P(Q)$ for some $Q\in GL_n(wL_K(E)w)$. Then, by item (3), $\phi_P\phi_Q = \phi_{_P\phi_P(Q)} = \phi_{PP^{-1}} = id_{L_K(E)}$, and so $\phi_P$ is surjective. By item (1), $\phi_P$ is always injective, and hence $\phi_P$ is an isomorphism with $\phi_{P}^{-1}=\phi_{Q}$.
	This implies that $id_{L_K(E)}= \phi^m_P\phi^m_Q = \phi_{P_m}\phi_{Q_m}= \phi_{_{P_{m}\phi_{P_{m}}(Q_m)}}$, so $\phi_{P_m}^{-1} = \phi_{Q_m}$ and $P_{m}\phi_{P_{m}}(Q_m)= wI_n$ for all $m\ge 1$. Consequently, $P_m = \phi_{P_m}(Q^{-1}_m)$ and $P^{-1}_m = \phi_{P_m}(Q_m)$ for all $m\ge 1$.
	
	In particular, suppose $\phi_P(P) = P$.  Since $\phi_P$ is a $K$-algebra homomorphism, $P\phi_P(P^{-1})= \phi_P(P)\phi_P(P^{-1}) = \phi_P(PP^{-1}) = \phi_P(wI_n) = wI_n$, so $P^{-1}=\phi_P(P^{-1})$. Similarly, we obtain that if $P^{-1}=\phi_P(P^{-1})$, then $P=\phi_P(P)$. Hence, in any case, we have that $P=\phi_P(P)$ and $P^{-1}=\phi_P(P^{-1})$. We then have $P^m \phi_P(P^{m}) = wI_n$ for all $m\in \mathbb{Z}$, so $\phi_P$ is an isomorphism and $\phi_P^{m} = \phi_{P^{m}}$ for all $m\in \mathbb{Z}$.

	(5) Assume that $|s^{-1}(v)| = n$ and let $\lambda: L_K(E)\to L_K(E)$ be a $K$-algebra homomorphism  with $\lambda(u)=u, \lambda(e)=e\text{ and } \lambda(e^*)=e^*$  for all $u\in E^0$ and $e\in E^1\setminus \{ e_1,\ldots ,e_n\}$. We then have $$\lambda(e_i)= \lambda(e_iw)= \lambda(e_i)\lambda(w)=\lambda(e_i)w$$ and 
	$$\lambda(e^*_i)= \lambda(we^*_i)= \lambda(w)\lambda(e^*_i)=w\lambda(e^*_i)$$ for all $1\le i\le n$, so $e^*_i\lambda(e_j)$ and $\lambda(e^*_i)e_j\in wL_K(E)w$ for all $1\le i\le n$.
	
	Let $P = (p_{i,j})$ and $P' = (p'_{i,j})\in M_n(wL_K(E)w)$ with $p_{i, j} = e^*_i\lambda(e_j)$ and $p'_{i, j} = \lambda(e^*_i)e_j$ for all $1\le i, j\le n$. We claim that $P\in GL_n(wL_K(E)w)$ with $P^{-1} = P'$. Indeed, since $|s^{-1}(v)|=n$, we must have $s^{-1}(v) = \{e_1, e_2, \ldots, e_n\}$ and $v = \sum^n_{i=1}e_ie^*_i$, and so 
	\[\sum^n_{k=1}p_{i,k}p'_{k,j} = \sum^n_{k=1}e^*_i\lambda(e_k) \lambda(e^*_k)e_j = e^*_i\lambda(\sum^n_{k=1}e_ke^*_k)e_j = e^*_i\lambda(v)e_j = \delta_{i, j}w\] and \[\sum^n_{k=1}p'_{i,k}p_{k,j} = \sum^n_{k=1}\lambda(e^*_i)e_ke^*_k\lambda(e_j) = \lambda(e^*_i)(\sum^n_{k=1}e_ke^*_k)\lambda(e_j) = \lambda(e^*_ie_j) = \delta_{i, j}w\] for all $1\le i, j\le n$,  where $\delta$ is the Kronecker delta. This implies that $PP' = wI_n = P'P$, showing the claim.
	
	We show that $\lambda = \phi_P$. It suffices to check that $\lambda(e_i) = \phi_P(e_i)$ and $\lambda(e^*_i) = \phi_P(e^*_i)$ for all $1\le i\le n$. For each $1\le i\le n$, by definition of $\phi_P$, we have $$\phi(e_i)= \sum^n_{k=1}e_ke^*_k\lambda(e_i) =(\sum^n_{k=1}e_ke^*_k)\lambda(e_i) = v\lambda(e_i)=\lambda(ve_i)=\lambda(e_i)$$ and $$\phi(e^*_i)= \sum^n_{k=1}\lambda(e^*_i)e_ke^*_k = \lambda(e^*_i)(\sum^n_{k=1}e_ke^*_k) = \lambda(e_i)v=\lambda(e^*_iv)=\lambda(e^*_i),$$ as desired.
	
	(6) We always have that $(GL_n(wL_K(E)w), \star)$ is a monoid with identity element $wI_n$. Then, the statement immediately follows from items (1), (2), (3) and (5), thus finishing the proof.
\end{proof}

Consequently, we obtain a method to construct graded endomorphisms and graded automorphisms of Leavitt path algebras of arbitrary graphs over an arbitrary field in terms of general linear groups over corners of these algebras.

\begin{cor}\label{gr-Iso}
	Let $K$ be a field, $n$ a positive integer, $E$ a graph, and $v$ and $w$ vertices in $E$ (they may be the same). Let $e_1, e_2, \ldots, e_n$ be distinct edges in $E$ with $s(e_i) = v$ and $r(e_i) = w$ for all $1\le i\le n$.  Let $P$ be an element of $GL_n(wL_K(E)_0w)$ with $P=(p_{i,j})$ and $P^{-1}=(p'_{i,j})$. Then the following statements hold:
	
	$(1)$ There exists a unique graded homomorphism $\phi_P:L_K(E)\longrightarrow L_K(E)$ of $K$-algebras satisfying
	\begin{center}
		$\phi _P(u)=u,\quad \phi _P(e)=e\quad\text{ and }\quad \phi _P(e^*)=e^*$ \end{center} for all $u\in E^0$ and $e\in E^1\setminus \{ e_1,\ldots ,e_n\} $, and 
	\begin{center}
		$\phi_P(e_i) = \sum^n_{k=1}e_kp_{k,i}$\quad	and \quad $\phi_P(e^*_i) = \sum^n_{k=1}p'_{i,k}e^*_k$
	\end{center} for all $1\le i\le n$.
	
	$(2)$ $\phi_P$ is a graded isomorphism if and only if $P^{-1} = \phi_P(Q)$ for some $Q\in GL_n(wL_K(E)_0w)$. In this case, $\phi_{P_m}^{-1} = \phi_{Q_m}$, $P_m = \phi_{P_m}(Q^{-1}_m)$ and $P^{-1}_m = \phi_{P_m}(Q_m)$ for all $m\ge 1$. In particular, if $\phi_{P}(P)= P$ or $\phi_{P}(P^{-1})= P^{-1}$, then $\phi_{P}$ is a graded isomorphism and $\phi_{P}^{m}=\phi_{P^{m}}$ for all integer $m$.
	
	$(3)$  Assume that $|s^{-1}(v)| = n$ and we denote by $End^{gr}_{v,w}(L_K(E))$ the set of all graded endomorphisms $\lambda$ of $L_K(E)$ with  $\lambda(u)=u,\ \lambda(e)=e \text{ and } \lambda(e^*)=e^*$  for all $u\in E^0$ and $e\in E^1\setminus \{ e_1,\ldots ,e_n\}$. Then, the map $\Phi: (GL_n(wL_K(E)_0w), \star)\longrightarrow End^{gr}_{v,w}(L_K(E))$, $P\longmapsto \phi_P$, is a monoid isomorphism, where the multiplication law ``$\star$" is defined by $$P\star Q = P\phi_P(Q)$$ for all $P, Q\in GL_n(wL_K(E)_0w)$.
\end{cor}
\begin{proof} 
	(1) By Theorem \ref{Iso}, there exists a unique homomorphism $\phi_{P}: L_K(E)\longrightarrow L_K(E)$ of $K$-algebras satisfying $\phi_{P}(u)=u, \phi_{P}(e)=e\text{ and } \phi_{P}(e^*)=e^*$  for all $u\in E^0$ and $e\in E^1\setminus \{ e_1,\ldots ,e_n\} $, and 
	\begin{center}
		$\phi_{P}(e_i) = \sum^n_{k=1}e_kp_{k,i}$\quad	and \quad $\phi_{P}(e^*_i) = \sum^n_{k=1}p'_{i,k}e^*_k$
	\end{center} for all $1\le i\le n$.  It is obvious that $\phi_P(u)$ has degree $0$ for all $u\in E^0$. Since $p_{i,j}$ and $p'_{i, j}\in L_K(E)_0$ for all $1\le i, j\le n$, $\phi_P(e)$ has degree $1$ and $\phi_P(e^*)$ has degree $-1$ for all $e\in E^1$. Therefore, $\phi_P$ is a $\mathbb{Z}$-graded homomorphism.
	
	(2) It immediately follows from Theorem \ref{Iso} (4).
	
(3) We note that for all $P, Q\in GL_n(wL_K(E)_0w)$, we obtain that $\phi_P(Q)\in GL_n(wL_K(E)_0w)$ (by item (1)) and 
$P\star Q = P\phi_P(Q)\in GL_n(wL_K(E)_0w)$, and so $GL_n(wL_K(E)_0w)$ is a submonoid of the monoid $(GL_n(wL_K(E)w), \star)$. Then, by Theorem \ref{Iso} (6), the map $\Phi: (GL_n(wL_K(E)_0w), \star)\longrightarrow End^{gr}_{v,w}(L_K(E))$, $P\longmapsto \phi_P$, is a monoid injection. 
	
We claim that $\Phi$ is surjective. Indeed, let $\lambda \in  End^{gr}_{v,w}(L_K(E))$. Then, by Theorem \ref{Iso} (5), there exists a unique matrix $P = (p_{i,j})\in GL_n(wL_K(E)w)$ such that $p_{i, j} = e^*_i\lambda(e_j)$ for all $1\le i, j\le n$ and $\lambda = \phi_P$. Since $\lambda$ is a graded homomorphism, $\lambda(e_j)$ has degree $1$ for all $1\le j\le n$, and so $p_{i, j} = e^*_i\lambda(e_j)\in L_K(E)_0$ for all $1\le i, j\le n$. This implies that $P\in GL_n(wL_K(E)_0w)$ and $\Phi(P) = \phi_P = \lambda$, showing the claim. Therefore, we have that $\Phi$ is a monoid isomorphism, thus finishing the proof.
\end{proof}

For clarification, we illustrate Theorem \ref{Iso} and Corollary~\ref{gr-Iso} by presenting the following example, which describes completely all (graded) endomorphisms and (graded) automorphism of the Levitt path algebra of the rose graph $R_1$ with one petal. 

\begin{exas}
	Let $K$ be a field and $R_1$ be the following graph. $$R_1 = 
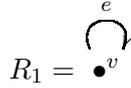
\xymatrix{\bullet^{v}\ar@(ul,ur)^e}$$ Then $L_K(R_1)\cong K[x, x^{-1}]$ via an isomorphism that sends $v$ to $1$, $e$ to $x$ and $e^*$ to $x^{-1}$. We then have that the group $U(L_K(R_1))$ of units of $L_K(R_1)$ is exactly the set $\{ae^m\mid a\in K\setminus\{0\}, m\in \mathbb{Z}\}$. For any $P = ae^m\in U(L_K(R_1))$, by Theorem \ref{Iso} (1), we have the endomorphism $\phi_P$ defined by: $v\longmapsto v$, $e\longmapsto ae^{m+1}$ and $e^*\longmapsto a^{-1}e^{-m-1}$. By Theorem \ref{Iso} (6), $End(L_K(R_1))$ is exactly the set $\{\phi_P\mid P\in U(L_K(R_1))\}$. We note that $a^{-1}e^{-m} = P^{-1} = \phi_P(be^l)$ if and only if $m=l =0$ and $b= a^{-1}$, or $m= l=-2$ and $b= a$. By Theorem \ref{Iso} (4), the automorphism group $Aut(L_K(R_1))$ of $L_K(R_1)$ is exactly the set $\{\phi_{a}, \phi_{be^{-2}}\mid a, b\in K\setminus \{0\}\}$.
	
	We have that $L_K(R_1)_0 = K$, and so  $End^{gr}(L_K(R_1))$ is exactly the set $\{\phi_a\mid a\in K\setminus \{0\}\}$ (by Corollary \ref{gr-Iso} (1)), which is isomorphic to the group $ K\setminus \{0\}$. We also have that $Aut^{gr}(L_K(R_1))$  is equal to $End^{gr}(L_K(R_1))$.
\end{exas}

\noindent The next aim of this section is to completely describe (graded) endomorphisms and (graded) automorphisms of the Leavitt algebra of type $(1;n)$ in terms of the general linear group of degree $n$ over this algebra.

Let $K$ be a field and $n \ge 2$ any integer. Then the  \textit{Leavitt $K$-algebra of type} $(1;n)$, denoted by $L_K(1, n)$,  is the $K$-algebra 
\begin{center}
$K\langle x_1, \hdots, x_n, y_1, \hdots, y_n\rangle/\langle \sum^{n}_{i=1}x_iy_i -1, y_ix_j - \delta_{i,j}1\mid 1\le i, j\le n \rangle.$\end{center}

 Notationally, it is often more convenient to view $L_K(1, n)$  as the free associative $K$-algebra on the $2n$ variables $x_1, \hdots, x_n, y_1, \hdots, y_n$ subject to the relations $ \sum^{n}_{i=1}x_iy_i =1$ and $y_ix_j = \delta_{i,j}1 \, (1\le i, j\le n)$; see \cite{leav:tmtoar} for more details.

For any integer $n\ge 2$, we let $R_n$ denote the \textit{rose graph with $n$ petals} having one vertex and $n$ loops:
$$R_n = 
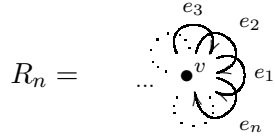
\xymatrix{ & {\bullet^v} \ar@(ur,dr)^{e_1}  \ar@(u,r)^{e_2} \ar@(ul,ur)^{e_3}  \ar@{.} @(l,u) \ar@{.} @(dr,dl)
	\ar@(r,d)^{e_n}  \ar@{}[l] ^{\hdots} } \ \ $$ Then $L_K(R_n)$ 	is defined to be the $K$-algebra generated by $v$, $e_1, \ldots, e_n$, $e^*_1, \ldots, e^*_n$, satisfying the following relations 
\begin{center}
	$v^2 =v, ve_i = e_i= e_iv$, $ve^*_i = e^*_i = e^*_i v$, $e^*_i e_j = \delta_{i,j} v$ and $\sum^n_{i=1}e_ie^*_i =v$	
\end{center}
for all $1\le i, j\le n$. In particular $v = 1_{L_K(R_n)}$. 

\begin{rem}
By \cite[Proposition 1.3.2]{AAS:LPA} (see, also \cite[Proposition 2.6]{kn:ataanirolpa}), $L_K(1, n)\cong L_K(R_n)$ as $K$-algebras, by the mapping: $1\longmapsto v$, $x_i\longmapsto e_i$	and $y_i\longmapsto e^*_i$ for all $1\le i\le n$. 
{\it With this fact in mind, for the remainder of this article we investigate (graded) automorphisms of the Leavitt algebra $L_K(1, n)$ by equivalently investigating (graded) automorphisms of the Leavitt path algebra $L_K(R_n)$}.
\end{rem}

\noindent The following proposition describes completely endomorphisms and automorphisms of $L_K(R_n)$ in terms of the general linear group of degree $n$ over $L_K(R_n)$.

\begin{prop}\label{IsoofLA}
	Let $n\ge 2$ be a positive integer, $K$ a field and $R_n$ the rose graph with $n$ petals. Let $P$ be an element of $GL_n(L_K(R_n))$ with $P=(p_{i,j})$ and $P^{-1}=(p'_{i,j})$. Then the following statements hold:
	
	$(1)$ There exists a unique injective homomorphism $\phi_{P}: L_K(R_n)\longrightarrow L_K(R_n)$ of $K$-algebras satisfying 
	$\phi_{P}(v)=v$, $\phi_{P}(e_i) = \sum^n_{k=1}e_kp_{k,i}$	and $\phi_{P}(e^*_i) = \sum^n_{k=1}p'_{i,k}e^*_k$ for all $1\le i\le n$. 
	
	$(2)$ $\phi_P\phi_Q = \phi_{_P\phi_P(Q)}$ for all $Q\in GL_n(L_K(R_n))$. In particular, $\phi^m_P = \phi_{P_m}$ for all positive integer $m$. 
	
	$(3)$ $\phi_P\in Aut(L_K(R_n))$ if and only if $P^{-1} = \phi_P(Q)$ for some $Q\in GL_n(L_K(R_n))$. In this case, $\phi_{P_m}^{-1} = \phi_{Q_m}$, $P_m = \phi_{P_m}(Q^{-1}_m)$ and $P^{-1}_m = \phi_{P_m}(Q_m)$ for all $m\ge 1$. In particular, if $\phi_{P}(P)= P$ or $\phi_{P}(P^{-1})= P^{-1}$, then $\phi_{P}$ is an isomorphism and $\phi_{P}^{m}=\phi_{P^{m}}$ for all integer $m$.
	
	$(4)$ The map $\Phi: (GL_n(L_K(R_n)), \star)\longrightarrow End(L_K(R_n))$, $P\longmapsto \phi_P$, is a monoid isomorphism, where the multiplication law ``$\star$" is defined by $$P\star Q = P\phi_P(Q)$$ for all $P, Q\in GL_n(L_K(R_n))$.
\end{prop}
\begin{proof}
	It immediately follows from Theorem \ref{Iso}.
\end{proof}

The following proposition describes completely graded endomorphisms and graded automorphisms of $L_K(R_n)$ in terms of the general linear group of degree $n$ over $L_K(R_n)_0$.

\begin{prop}\label{gr-IsoofLA}
	Let $n\ge 2$ be a positive integer, $K$ a field and $R_n$ the rose graph with $n$ petals. Let $P$ be an element of $GL_n(L_K(R_n)_0)$ with $P=(p_{i,j})$ and $P^{-1}=(p'_{i,j})$. Then the following statements hold:
	
	$(1)$ There exists a unique graded homomorphism $\phi_{P}: L_K(R_n)\longrightarrow L_K(R_n)$ of $K$-algebras satisfying 
	$\phi_{P}(v)=v$, $\phi_{P}(e_i) = \sum^n_{k=1}e_kp_{k,i}$	and $\phi_{P}(e^*_i) = \sum^n_{k=1}p'_{i,k}e^*_k$ for all $1\le i\le n$. 
	
	$(2)$ $\phi_P\in Aut^{gr}(L_K(R_n))$ if and only if there exists a matrix $Q\in GL_n(L_K(R_n)_0)$ such that $P^{-1} = \phi_P(Q)$. In this case, $\phi_{P_m}^{-1} = \phi_{Q_m}$,  $P_m = \phi_{P_m}(Q^{-1}_m)$ and $P^{-1}_m = \phi_{P_m}(Q_m)$ for all $m\ge 1$. In particular, if $\phi_{P}(P)= P$ or $\phi_{P}(P^{-1})= P^{-1}$, then $\phi_{P}$ is a graded isomorphism and $\phi_{P}^{m}=\phi_{P^{m}}$ for all integer $m$.
	
	$(3)$ The map $\Phi: (GL_n(L_K(R_n)_0), \star)\longrightarrow End^{gr}(L_K(R_n))$, $P\longmapsto \phi_P$, is a monoid isomorphism, where the multiplication law ``$\star$" is defined by $$P\star Q = P\phi_P(Q)$$ for all $P, Q\in GL_n(L_K(R_n)_0)$.
\end{prop}
\begin{proof}
	It immediately follows from Corollary \ref{gr-Iso}.
\end{proof}


The following corollary gives that the general linear group $GL_n(K)$ of degree $n$ over a field $K$ may be considered as a subgroup of the graded automorphism group $Aut^{gr}(L_K(R_n))$ of $L_K(R_n)$.

\begin{cor}\label{gl-IsoofLA}
Let $n\ge 2$ be a positive integer, $K$ a field and $R_n$ the rose graph with $n$ petals. Then, there exists an injective homomorphism $\Phi: GL_n(K)\longrightarrow Aut^{gr}(L_K(R_n))$ of groups such that $\Phi(P) = \phi_P$ for all $P\in GL_n(K)$.	
\end{cor}
\begin{proof}
By proposition \ref{gr-IsoofLA} (3), 	the map $\Phi: (GL_n(L_K(R_n)_0), \star)\longrightarrow End^{gr}(L_K(R_n))$, defined by $P\longmapsto \phi_P$, is a monoid isomorphism, where the multiplication law ``$\star$" is defined by $$P\star Q = P\phi_P(Q)$$ for all $P, Q\in GL_n(L_K(R_n)_0)$. For all $P$ and $Q\in GL_n(K)$, since $\phi_P(Q) = Q$, we must have $P\star Q = PQ$, so $GL_n(K)$ is a subgroup of the group of units of the monoid $(GL_n(L_K(R_n)_0), \star)$. Moreover, since $\phi_P(P) = P$ for all $P\in GL_n(K)$, and by Proposition \ref{gr-IsoofLA}, $\phi_P\in Aut^{gr}(L_K(R_n))$ for all $P\in GL_n(K)$. From these observations, we obtain that $\Phi|_{GL_n(K)}: GL_n(K)\longrightarrow Aut^{gr}(L_K(R_n))$ is an injective homomorphism of groups, thus finishing the proof.
\end{proof}

\noindent In \cite[Corollary 2.8]{kn:ataanirolpa} Kuroda and the first author introduced Anick type automorphisms of $L_K(R_n)$. We reproduce here these automorphisms. Namely, 
for any integer $n\ge 2$ and any field $K$, we denote by $A_{R_n}(e_1, e_2)$ the $K$-subalgebra of $L_K(R_n)$ generated by $$v, e_1, e_3, \hdots , e_n, e^*_2, \hdots , e^*_n.$$	

We should note that by \cite[Theorem 1]{aajz:lpaofgkd} (see, also \cite[Theorem 3.7]{ln:othdolpawcicr}), the following elements form a basis of the $K$-algebra $A_{R_n}(e_1, e_2)$: (1) $v$, (2) $p= e_{k_1}\cdots e_{k_m}$, where $k_i \in \{1, 3, \hdots , n\}$, (3) $q^* = e^*_{t_1}\cdots  e^*_{t_h}$, where $t_i \in \{2, 3, \hdots , n\}$,
(4) $pq^*$, where $p$ and $q^*$ are defined as in items (2) and (3), respectively. 

For any $p\in A_{R_n}(e_1, e_2)$, let 
$$U_p=\begin{pmatrix}
1 & p & 0 & \dots & 0 \\
0 & 1 & 0 & \dots & 0 \\
\vdots & \vdots &\vdots &\vdots &\vdots \\
0 & 0 & 0 & \dots & 1 \\
\end{pmatrix}.$$ We then have $U_p\in GL_n(L_K(R_n))$ with $U^{-1}_p = U_{-p}$ and $$U_pU_q = U_{p+q}$$ for all $p, q\in A_{R_n}(e_1, e_2)$. Also, for any $p\in A_{R_n}(e_1, e_2)$, by Theorem \ref{Iso}, we obtain the endomorphism $\phi_{U_p}$ of $L_K(R_n)$ defined by: $v\longmapsto v$, $e_i\longmapsto e_i$ for all $i\in \{1, 3, \ldots, n\}$, $e^*_j\longmapsto e^*_j$ for all $2\le j\le n$, $e_2\longmapsto e_2 + e_1p$ and $e^*_1\longmapsto e^*_1 - pe^*_2$. We note that $\phi_{U_p}(q) = q$ for all $q\in A_{R_n}(e_1, e_2)$, and so $\phi_{U_p}(U_q) = U_q$ for all $q\in A_{R_n}(e_1, e_2)$. By Theorem \ref{Iso}, $\phi_{U_p}$ is an automorphism and $\phi^m_{U_p} = \phi_{{U_p}_m}$ for all $p\in A_{R_n}(e_1, e_2)$ and $m\in \mathbb{Z}$. Moreover, if $p\in A_{R_n}(e_1, e_2)\cap L_K(R_n)_0$, then $\phi_{U_p}$ is a graded automorphism by Proposition \ref{gr-IsoofLA}. From these observations, we have the following interesting note.

\begin{cor}\label{Anick-IsoofLA}
Let $n\ge 2$ be a positive integer, $K$ a field and $R_n$ the rose graph with $n$ petals. Then, there is an injective homomorphism $\Phi: (A_{R_n}(e_1, e_2), +)\longrightarrow Aut(L_K(R_n))$ of groups such that $\Phi(p) = \phi_{U_p}$ for all $p\in A_{R_n}(e_1, e_2)$, and $$\Phi|_{A_{R_n}(e_1, e_2)\cap L_K(R_n)_0}: (A_{R_n}(e_1, e_2)\cap L_K(R_n)_0, +)\longrightarrow Aut^{gr}(L_K(R_n))$$ is an injective homomorphism of groups
\end{cor}
\begin{proof}
By proposition \ref{IsoofLA} (4), 	the map $\Phi: (GL_n(L_K(R_n)), \star)\longrightarrow End(L_K(R_n))$, defined by $P\longmapsto \phi_{P}$, is a monoid isomorphism, where the multiplication law ``$\star$" is defined by $$P\star Q = P\phi_P(Q)$$ for all $P, Q\in GL_n(L_K(R_n))$. Since $\phi_{U_p}(U_q) = U_q$ for all  $p, q\in A_{R_n}(e_1, e_2)$, we have $$U_p\star U_q =U_p\phi_{U_p}(U_q)=U_pU_q=U_{p+q}$$ for all  $p, q\in A_{R_n}(e_1, e_2)$. This implies that the map from $(A_{R_n}(e_1, e_2), +)$ to the group of units of the monoid $(GL_n(L_K(R_n)), \star)$, defined by $p\longmapsto U_p$, is an injective homomorphism of groups. Hence, the group $(A_{R_n}(e_1, e_2), +)$ may be viewed as a subgroup of the group of units of the monoid $(GL_n(L_K(R_n)), \star)$, and so $$\Phi|_{A_{R_n}(e_1, e_2)}: (A_{R_n}(e_1, e_2), +)\longrightarrow Aut(L_K(R_n))$$ is an injective homomorphism of groups satisfying the desired statements, thus finishing the proof.
\end{proof}

Next, we give a complete description of all automorphisms of $L_K(R_n)$ via the group of units $U(L_K(R_n))$ of $L_K(R_n)$. To do so, we need the following useful remark.

\begin{rem}\label{Unit-Ggrg}
Let $n\ge 2$ be a positive integer, $K$ a field and $R_n$ the rose graph with $n$ petals. Then, since $L_K(R_n)\cong L_K(R_n)^n$ as left $L_K(R_n)$-modules, we immediately obtain that $L_K(R_n)\cong M_n(L_K(R_n))$ as $K$-algebras. The isomorphism and its inverse are, respectively, easy to write down explicitly:
\begin{center}
$s\longmapsto (e^*_ise_j)$\quad and\quad $M = (m_{i,j})\longmapsto \sum_{1\le i, j\le n}e_im_{i,j}e^*_j.$	
\end{center}	
\end{rem}

Using Theorem \ref{Iso}, Proposition \ref{IsoofLA} and Remark \ref{Unit-Ggrg}, we obtain the following interesting corollary, which was studied by Cuntz in \cite{cuntz:aocaCa}.  

\begin{cor}\label{Cuntz-IsoofLA}
Let $n\ge 2$ be a positive integer, $K$ a field, $R_n$ the rose graph with $n$ petals, $U(L_K(R_n))$ the group of units of $L_K(R_n)$, and $u$ an element of $U(L_K(R_n))$. Then the following statements hold:

$(1)$ The map $f_u: L_K(R_n)\longrightarrow L_K(R_n)$, defined by $v\longmapsto v$, $e_i\longmapsto ue_i$ and $e^*_i\longmapsto u^{-1}e^*_i$ for all $1\le i\le n$, is an injective homomorphism of $K$-algebras, and $f_u = \phi_P$, where $P = (e^*_iue_j)\in GL_n(L_K(R_n))$ and $\phi_P$ is the endomorphism of $L_K(R_n)$ introduced in Proposition \ref{IsoofLA}.

$(2)$ For every $\lambda \in End(L_K(R_n))$, $\lambda =f_x$, where $x = \sum_{i=1}^n\lambda(e_i)e^*_i\in U(L_K(R_n))$. In particular, if $\tau_u$ is the inner automorphism of $L_K(R_n)$ generated by $u$, then $\tau_u = f_x$, where $x = u^{-1}(\sum^n_{i=1}e_iue^*_i)$.

$(3)$ $f_u f_w = f_{f_u(w)u}$ for all $w\in U(L_K(R_n))$.

$(4)$ $f_u\in Aut(L_K(R_n))$ if and only if $u^{-1} = f_u(w)$ for some $w\in U(L_K(R_n))$. In this case, $f_u^{-1} = f_{w}$. Consequently, $$Aut^{gr}(L_K(R_n)) = \{f_u \in Aut(L_K(R_n)) \mid u\in U(L_K(R_n)_0)\}.$$

$(5)$ The map $\Phi: (U(L_K(R_n)), \star)\longrightarrow End(L_K(R_n))$, $u\longmapsto f_u$, is a monoid isomorphism, where the multiplication law ``$\star$" is defined by $$u\star w = f_u(w)u$$ for all $u, w\in U(L_K(R_n))$.

$(6)$ Let $Inn(L_K(R_n))$ be the inner automorphism group of $L_K(R_n)$. Then the canonical homomorphism $\mathcal{T}: U(L_K(R_n))\longrightarrow Inn(L_K(R_n))$, $u\longmapsto \tau_u$, is surjective with $\ker(\mathcal{T}) = K\cdot 1_{L_K(R_n)}$. Consequently, $Z(U(L_K(R_n))) = K\cdot 1_{L_K(R_n)}$.
\end{cor}
\begin{proof}
(1)	Since $u\in U(L_K(R_n))$ and by Remark \ref{Unit-Ggrg}, $P := (e^*_iue_j) \in GL_n(L_K(R_n))$ with $P^{-1} = (e^*_iu^{-1}e_j)$. By Proposition \ref{IsoofLA} (1), $\phi_P$ is an injective $K$-algebra endomorphism of $L_K(R_n)$. Also, we have $\phi_P(e_i) = \sum^n_{k=1}e_ke^*_kue_i = ue_i=f_u(e_i)$ and $\phi_P(e^*_i) = \sum^n_{k=1}e^*_iu^{-1}e_ke^*_k = e^*_iu^{-1}=f_u(e^*_i)$ for all $1\le i\le n$, and so $f_u = \phi_P$ and consequently, $f_u$ is an injective $K$-algebra endomorphism of $L_K(R_n)$.

(2) Let $\lambda\in End(L_K(R_n))$. By Theorem \ref{Iso} (5), $\lambda = \phi_P$, where $P = (e^*_i\lambda(e_j))\in GL_n(L_K(R_n))$. On the other hand, by Item (1) and Remark \ref{Unit-Ggrg}, $\phi_P = f_x$, where $x = \sum_{1\le i, j\le n}e_ie^*_i\lambda(e_j)e^*_j = \sum^n_{i=1}\lambda(e_i)e^*_i\in U(L_K(R_n))$. Therefore, $\lambda =f_x$, as desired.

(3) Let $w\in U(L_K(R_n))$. We then have $f_uf_w = \phi_P\phi_Q= \phi_{P\phi_P(Q)}$, where $P = (e^*_iue_j)$ and $Q = (e^*_iwe_j)\in GL_n(L_K(R_n))$. Also, $P\phi_P(Q) = Pf_u(Q) = (e^*_if_u(w)ue_j)$ and $\sum_{1\le i, j\le n}e_ie^*_if_u(w)ue_je^*_j = f_u(w)u$. By Item (1) and Remark \ref{Unit-Ggrg}, $\phi_{P\phi_P(Q)} = f_{f_u(w)u}$, and so $f_uf_w = f_{f_u(w)u}$, as desired.

(4) We have that $f_u\in Aut(L_K(R_n))$ if and only if $\phi_P\in Aut(L_K(R_n))$, where $P = (e^*_iue_j)\in GL_n(L_K(R_n))$, if and only if $P^{-1} = (e^*_iu^{-1}e_j) = \phi_P(Q)= f_u(Q)$ for some $Q= (q_{i, j})\in GL_n(L_K(R_n))$, if and only if $Q = (f^{-1}_u(e^*_iu^{-1}e_j)) = (e^*_iwe_j)$, where $w = f^{-1}_u(u^{-1})\in U(L_K(R_n))$ (since $f_u$ is always injective). In this case, by Proposition \ref{IsoofLA} (3) and Item (1), $f^{-1}_u = \phi^{-1}_P = \phi_Q = f_w$. From these observations, we immediately obtain that $f_u\in Aut(L_K(R_n))$ if and only if $u^{-1} = f_u(w)$ for some $w\in U(L_K(R_n))$.

(5) It immediately follows from Items (1), (2), (3) and Proposition \ref{IsoofLA} (4).

(6) Let $\tau_u$ be the inner automorphism of $L_K(R_n)$ generated by $u$, for which $\tau_u = id_{L_K(R_n)}$. We then have $u^{-1}e_iu = e_i$ for all $1\le i \le n$. Equivalently, $u = e^*_iue_i$ for all $1\le i \le n$. This implies that $u = (e^*_i)^mu e^m_i$ for all $1\le i \le n$ and for all $m\ge 1$. Since $L_K(R_n)$ is a $\mathbb{Z}$-graded algebra, $u$ may be written in the form: $u = \sum^t_{j= -l}u_j$, where $l, t\ge 0$ and $u_j\in L_K(R_n)_j$ for all $-l \le j\le t$. We then have $$\sum^t_{j= -l}u_j =u= \sum^t_{j= -l}(e^*_i)^m u_j e^m_i$$ for all $1\le i\le n$ and $m\ge 0$. By $\mathbb{Z}$-grading in $L_K(R_n)$, we must have $u_j = (e^*_i)^m u_j e^m_i$ for all $-l \le j\le t$, $1\le i\le n$ and $m\ge 1$. For every $-l\le j\le t$, we write $u_j = \sum^h_{k=1}a_k\alpha_k\beta^*_k$, where $h\ge 0$, $a_k\in K$ and $\alpha_k, \beta_k\in (R_n)^*$ with $|\alpha_k|- |\beta_k| = j$. Let $m:= \max\{|\alpha_k|\mid 1\le k\le h\}$. For each $1\le i\le n$, we obtain that $$u_j = (e^*_i)^m u_j e^m_i = (e^*_i)^m (\sum^h_{k=1}a_k\alpha_k\beta^*_k) e^m_i = \sum^h_{k=1}a_k(e^*_i)^m\alpha_k\beta^*_ke^m_i = b_i e_i^j$$
for some $b_i\in K$, where $e^0_i =v$ and $e^j_i = (e^*_i)^{-j}$ if $j<0$. By $\mathbb{Z}$-grading in $L_K(R_n)$, $u_j = 0$ for all $j\neq 0$, and so $u = k v$ for some $k\in K\setminus\{0\}$. Therefore, we have $\ker(\mathcal{T}) = K\cdot 1_{L_K(R_n)}$ and $Z(U(L_K(R_n))) = K\cdot 1_{L_K(R_n)}$, thus finishing the proof.
\end{proof}

It is worth mentioning the following note.

\begin{rem}\label{rem-auto}
(1) We should note that Propositions \ref{IsoofLA} and  \ref{gr-IsoofLA} immediately follow as a consequence of Corollary \ref{Cuntz-IsoofLA} and Remark \ref{Unit-Ggrg}.

(2) Let $K$ be an arbitrary field,  $n\ge 2$ a positive integer,  and $R_n$ the rose with $n$ petals. Let $u\in U(L_K(R_n))$ and $P = (e^*_iue_j)\in GL_n(L_K(R_n))$. Let $f_u$ be the automorphism of $L_K(R_n)$ introduced in Corollary \ref{Cuntz-IsoofLA}. Then 
\begin{center}
$f_u(P) = P$ if and only if $f_u(u) = u$. 	
\end{center}
Indeed,  assume that $f_u(P)=P$, i.e., $f_u(e^*_iue_j)=e^*_iue_j$ for all $1\leq i,j\leq n$. This follows that $e^*_iu^{-1}f_u(u)ue_j=e^*_iue_j$ and 
$e_ie^*_iu^{-1}f_u(u)ue_je_j^*=e_ie^*_iue_je_j^*$
for all $1\leq i,j\leq n$. Hence, $$\sum_{i=1}^{n}\sum_{j=1}^n e_ie^*_iu^{-1}f_u(u)ue_je_j^*=\sum_{i=1}^{n}\sum_{j=1}^n e_ie^*_iue_je_j^*.$$ Equivalently, 
$$\left(\sum_{i=1}^{n}e_ie^*_i\right)u^{-1}f_u(u)u \left(\sum_{j=1}^n e_je_j^*\right) = \left(\sum_{i=1}^{n}e_ie^*_i\right) u \left(\sum_{j=1}^n e_je_j^*\right),$$
so $u^{-1}f_u(u)u=u$ and $f_u(u)=uuu^{-1}=u$.	 The converse is obvious.
\end{rem}

We close this section with the following example.

\begin{exas}\label{Example-auto}
Let $K$ be a field and $R_2$ the rose graph with $2$ petals.

(1) $P=\begin{pmatrix}
0 & v \\
v & 0 \\
\end{pmatrix}\in M_2(L_K(R_2))$. It is obvious that $P$ is an invertible matrix with $P^{-1} = P$. Let $x = e_1e^*_2 + e_2e^*_1\in L_K(R_n)_0$. Since $x$ is the image of $P$ under the displayed isomorphism given in Remark \ref{Unit-Ggrg}, $x$ is automatically a unit of $L_K(R_2)_0$ with $x^{-1} = x$.	Then, by Corollary \ref{Cuntz-IsoofLA} (1), we have the graded endomorphism $f_x$ of $L_K(R_2)$ such that $f_x(v) =v$, $f_x(e_1) = xe_1 = e_2$, $f_x(e_2) = xe_2 = e_1$, $f_x(e_1^*)=e_1^*x^{-1}=e_2^*$ and $f_x(e_2^*)=e_2^*x^{-1}=e_1^*$. It is obvious that $f_x(P) =P$, and so $f_x(x) = x$ (by Remark \ref{rem-auto} (2)). This implies that $f_x(x^{-1}) = x^{-1}$, and so $f_x$ is an automorphism of $L_K(R_2)$ by Corollary \ref{Cuntz-IsoofLA} (4).

(2) Let $A_{R_2}(e_1, e_2)$ is the $K$-subalgebra of $L_K(R_2)$ generated by $v, \, e_1, \, e^*_2$, that means, $$A_{R_2}(e_1, e_2) = \{\sum^n_{i =1}r_i e^{m_i}_1 (e^*_2)^{l_i}\mid n\ge 1,\, r_i \in K,\, m_i, l_i\ge 0 \},$$ where $e^0_1 = v = (e^*_2)^0$. Let $p := e_1e_2^*\in A_{R_2}(e_1, e_2)$ and $U_p=\begin{pmatrix}
v & p \\
0 & v \\
\end{pmatrix}\in M_2(L_K(R_2))$. As introduced prior to Corollary \ref{Anick-IsoofLA}, $U_p$ is invertible with $(U_p)^{-1} = U_{-p}$. Let $y = v + e_1^2(e^*_2)^2\in L_K(R_2)_0$. Since $y$ is the image of $U_p$ under the displayed isomorphism given in Remark \ref{Unit-Ggrg}, $y$ is automatically a unit of $L_K(R_2)_0$ with $y^{-1} = v-e_1^2(e^*_2)^2$. Then, by Corollary \ref{Cuntz-IsoofLA} (1), we have the graded endomorphism $f_y$ of $L_K(R_2)$ such that $f_y(v) =v$, $f_y(e_1) = ye_1 = e_1$, $f_y(e_2) = ye_2 = e_2 + e_1^2e^*_2$, $f_y(e_1^*)=e_1^*y^{-1}=e_1^* -e_1(e^*_2)^2$ and $f_y(e_2^*)=e_2^*y^{-1}=e_2^*$. By Corollary \ref{Cuntz-IsoofLA} (1), we have $f_y = \phi_{U_p}$, where $\phi_{U_p}$ is the automorphism of $L_K(R_2)$ introduced prior to Corollary \ref{Anick-IsoofLA}. We note that $f_y(q)= \phi_{U_p}(q) =q$ for all $q\in A_{R_2}(e_1, e_2)$, and so $f_y(y) = y$. This implies that $f_y(y^{-1}) = y^{-1}$, and so $f_y$ is an automorphism of $L_K(R_2)$ by Corollary \ref{Cuntz-IsoofLA} (4).\medskip

(3) Let $u := (e_1e^*_2 + e_2e^*_1)( v + e_1^2(e^*_2)^2) = e_1e_2^*+e_2e_1^*+e_1^2e_2^*e_1^*$. We have that $u$ is a unit of $L_K(R_2)_0$ with $u^{-1} = (v - e_1^2(e^*_2)^2)(e_1e^*_2 + e_2e^*_1) = e_1e_2^*+e_2e_1^*-e_2e_1(e_2^*)^2$. By Corollary \ref{Cuntz-IsoofLA} (1), we have the graded endomorphism $f_u$ of $L_K(R_2)$ such that $f_u(v) =v$, $f_u(e_1) = ue_1 = e_2+e_1^2e_2^*$, $f_u(e_2) = ue_2 = e_1$ $f_u(e_1^*)=e_1^*u^{-1}=e_2^*$ and $f_u(e_2^*)=e_2^*u^{-1}=e_1^*-e_1(e_2^*)^2$. We then obtain that
\begin{align*}
f_u(e_1e_2^*+e_2e_1^*-e_2^2e_1^*e_2^*)&=(e_2+e_1^2e_2^*)(e_1^*-e_1(e_2^*)^2)+e_1e_2^*-e_1^2e_2^*(e_1^*-e_1(e_2^*)^2)\\
&=e_2e_1^*-e_2e_1(e_2^*)^2+e_1^2e_2^*e_1^*+e_1e_2^*-e_1^2e_2^*e_1^*\\
&=e_1e_2^*+e_2e_1^*-e_2e_1(e_2^*)^2=u^{-1}.
\end{align*} By Corollary \ref{Cuntz-IsoofLA} (4), $f_u$ is a graded automorphism of $L_K(R_2)$ with $f^{-1}_u = f_w$, where $w= e_1e_2^*+e_2e_1^*-e_2^2e_1^*e_2^*\in U(L_K(R_2))$ with $w^{-1} = e_1e_2^*+e_2e_1^*+e_1e_2(e_1^*)^2$.
	
We note that 
$f_u(u)=f_u(e_1e_2^*+e_2e_1^*+e_1^2e_2^*e_1^*)
=(e_2+e_1^2e_2^*)(e_1^*-e_1(e_2^*)^2)+e_1e_2^*+(e_2+e_1^2e_2^*)^2(e_1^*-e_1(e_2^*)^2)e_2^*=e_1e_2^*+e_2e_1^*+e_1^2e_2^*e_1^*+e_1^2e_1^*e_2^*-e_2e_1(e_2^*)^2+e_2^2e_1^*e_2^*-e_2^2e_1(e_2^*)^3+e_2e_1^2e_2^*e_1^*e_2^*-e_1^3(e_2^*)^3.$ 
Therefore, by the $\mathbb{Z}$-grading in $L_K(R_2)$, we must have $f_u(u) \neq u$.\medskip

The above examples show that the set $$\{u\in U(L_K(R_2))\mid f_u(u) = u\}$$ is not a subgroup of $U(L_K(R_2))$.
\end{exas}

\section{Application: Zhang twist of Leavitt path algebras}	

\noindent In this section we study Zhang twist of Leavitt path algebras. More precisely, we twist the multiplicative structure of Leavitt path algebras $L_K(E)$ over any graph $E$ with the help of graded automorphisms constructed in the previous section. 

\begin{defn}
Let $\sigma$ be a graded automorphism of Leavitt path algebra $L_K(E)$ over any arbitrary graph $E$. We know that $L_K(E)$ has a $\mathbb Z$-graded structure as $L_K(E)=\oplus_n L_n$. We twist the multiplicative structure of $\oplus_n L_n$ as $a\star b=a\sigma^n(b)$ for any $a\in L_n, b\in L_m$. The same underlying graded vector space $\oplus L_n$ with this new graded product $\star$ is called the Zhang twist of $L_K(E)$ and denoted as $L_K(E)^{\sigma}$. 
\end{defn}

\noindent In a rather surprising result we note that the Leavitt path algebra $L_K(E)$ of an arbitrary graph $E$ is always a subalgebra of the Zhang twist $L_K(E)^{\phi_P}$ by any graded automorphism $\phi_P$ introduced in Corollary \ref{gr-Iso}. 

\begin{prop}\label{ZhangTw-SubalLPA}
	Let $K$ be a field, $n$ a positive integer, $E$ a graph, and $v$ and $w$ vertices in $E$ (they may be the same). Let $e_1, e_2, \ldots, e_n$ be distinct edges in $E$ with $s(e_i) = v$ and $r(e_i) = w$ for all $1\le i\le n$.  Let $P=(p_{ij})$ and $Q=\left(q_{ij}\right)$ be elements of  $GL_n(wL_K(E)_0w)$ with $P\varphi_P(Q) =I_n$, $P^{-1}= (p^{(-1)}_{ij})$ and $Q^{-1}= (q^{(-1)}_{ij})$. Then, there exists a graded injective homomorphism $\theta_{P}: L_K(E)\longrightarrow L_K(E)^{\phi_P}$ of $K$-algebras satisfying 
	\begin{center}
		$\theta_{P}(u)=u,\quad \theta_{P}(e)=e\quad\text{ and }\quad \theta_{P}(f^*)=f^*$ \end{center} for all $u\in E^0$, $e\in E^1$ and $f\in E^1\setminus \{ e_1,\ldots ,e_n\} $, and 
	\begin{center}
		$\theta_{P}(e^*_i) = \sum^n_{k=1}q^{(-1)}_{ik}e^*_k$
	\end{center} for all $1\le i\le n$, where the graded automorphism $\phi_P$ is defined in Corollary \ref{gr-Iso}.
\end{prop}
\begin{proof}
	We first note that $\phi_{P}(u)=u, \phi_{P}(e)=e\text{ and } \phi_{P}(e^*)=e^*$  for all $u\in E^0$ and $e\in E^1\setminus \{ e_1,\ldots ,e_n\} $, and 
	\begin{center}
		$\phi_{P}(e_i) = \sum^n_{k=1}e_kp_{ki}$\quad	and \quad $\phi_{P}(e^*_i) = \sum^n_{k=1}p^{(-1)}_{ik}e^*_k$
	\end{center} for all $1\le i\le n$, and $\phi^{-1}_P = \phi_Q$. 
	
	We define the  elements $\{Q_u \ | \ u\in E^0\}$ and $\{T_e, T_{e^*} \ | \ e\in E^1\}$ of $L_K(E)^{\phi_P}$  by setting $Q_u = u$, $T_e = e$	and 
	\begin{equation*}
	T_{e^*}=  \left\{
	\begin{array}{lcl}
	\sum^n_{k=1}q^{(-1)}_{ik}e^*_k&  & \text{if } e= e_i \text{ for some } 1\le i\le n\\
	e^*&  & \text{otherwise}.
	\end{array}%
	\right.
	\end{equation*}%
	\medskip
	
	\noindent
	We claim that $\{Q_u, T_e, T_{e^*}\mid u\in E^0, e\in E^1\}$ is a family in $L_K(E)^{\phi_P}$ satisfying the relations analogous to (1) - (4) in Definition~\ref{LPA}. Indeed, we have $Q_u\ast Q_{u'}=Q_u Q_{u'} = u u' = \delta_{u, u'}u = \delta_{u, u'} Q_u$ for all $u, u'\in E^0$, showing relation (1).
	
	For (2), we always have $Q_{s(e)}\ast T_e=Q_{s(e)}T_e = T_e = T_eT_{r(e)}= T_e\ast T_{r(e)}$ for all $e\in E^1$ and $T_{f^*}\ast Q_{s(f)}=T_{f^*} \phi^{-1}(Q_{s(f)}) = T_{f^*}Q_{s(f)} =T_{f^*} = Q_{r(f)}T_{f^*}=Q_{r(f)}\ast T_{f^*}$ for all $f\in E^1\setminus \{e_1,\ldots ,e_n\}$. For each $1\le i\le n$, since 
	$$ve_k=e_kw=e_k,\ \ we_k^*=e_k^*v=e_k^*,\ \  \text{and}\ \ wq^{(-1)}_{ik}=q^{(-1)}_{ik}$$ for all $k$, we have 
	\begin{align*}
	Q_w\ast T_{e^*_i}&=Q_wT_{e^*_i}=w\sum _{k=1}^nq^{(-1)}_{ik}e_k^*
	=\sum _{k=1}^nwq^{(-1)}_{ik}e_k^*
	=\sum _{k=1}^nq^{(-1)}_{ik}e_k^*=T_{e^*_i},\\
	T_{e^*_i}\ast Q_v&=T_{e^*_i} \phi^{-1}_P(Q_v)=T_{e^*_i}Q_v=\sum _{k=1}^nq^{(-1)}_{ik}e_k^*v
	=\sum _{k=1}^nq^{(-1)}_{ik}e_k^*=T_{e^*_i}. 
	\end{align*}
	
	For (3), we obtain that $T_{e^*}\ast T_f = e^* \phi^{-1}_P(f) =e^* \phi_Q(f)=e^*f= \delta_{e, f}r(e)$ for all $e, f\in E^1\setminus \{ e_1,\ldots ,e_n\}$. For each $f\in E^1\setminus \{ e_1,\ldots ,e_n\} $ and $1\le i\le n$, we have 
	$$T_{e^*_i}\ast T_{f}=T_{e^*_i}\phi^{-1}_P(T_{f})=T_{e^*_i}\phi_Q(T_{f})=\sum _{k=1}^nq_{ik}e_k^*f=0$$ and $$
	T_{f^*}\ast T_{e_i}=T_{f^*}\phi^{-1}_P(T_{e_i})=T_{f^*}\phi_Q(T_{e_i})=\sum _{k=1}^nf^*e_kp_{ki}=0,$$	
	since $e_k^*f=f^*e_k=0$. For $i,j\in \{ 1,\ldots ,n\}$, we have 
	\begin{align*}
	T_{e^*_i}\ast T_{e_j}&=T_{e^*_i}\phi^{-1}_P(T_{e_j})=T_{e^*_i}\phi_Q(T_{e_j})=\sum _{k=1}^n\sum _{l=1}^nq^{(-1)}_{ik}e_k^*e_lq_{lj}\\
	&=\sum _{k=1}^n\sum _{l=1}^nq^{(-1)}_{ik}\delta _{k,l}wp_{lj} 
	=\sum _{k=1}^nq^{(-1)}_{ik}p_{kj}=\delta_{i,j}w =\delta_{i,j}Q_w, 
	\end{align*}
	since $e_k^*e_l=\delta _{k,l}w$ and $wp_{lj}=p_{lj}$. 
	
	For (4), let $u$ be a regular vertex in $E$. If $u \neq v$, then $\sum_{e\in s^{-1}(u)}T_e\ast T_{e^*} = \sum_{e\in s^{-1}(u)}T_e\phi_P(T_{e^*}) = \sum_{e\in s^{-1}(u)}ee^* = u = Q_u$. Consider the case when $u=v$, that is, $v$ is a regular vertex. Write $$s^{-1}(v)=\{ e_1,\ldots ,e_n,e_{n+1},\ldots, e_m\}$$
	for some distinct $e_{n+1},\ldots ,e_m\in E^1$ with $n\le m<\infty $. We note that $T_{e_k}\ast T_{e^*_k} = T_{e_k}\phi_P(T_{e^*_k}) = e_ke^*_k$ for all $n+1\le k\le m$, and
	\begin{align*}
	T_{e_i}\ast T_{e^*_i}&=e_i \varphi_P(\sum_{k=1}^{n}q^{(-1)}_{ik}e_k^*)=e_i\sum_{k=1}^{n}\varphi_P(q^{(-1)}_{ik}) \varphi_P(e_k^*)\\
	&=e_i\sum_{k=1}^{n}p_{ik}(\sum_{t=1}^{n}p^{(-1)}_{kt}e_t^*) \quad \textnormal{(since $\varphi_P(Q^{-1})=P$)}\\
	&=e_i\sum_{t=1}^{n}(\sum_{k=1}^{n}p_{ik}p^{(-1)}_{kt})e_t^*
	=e_i (\sum_{k=1}^{n}p_{ik}p^{(-1)}_{ki}) e_i^*=e_iwe_i^*\\
	&=e_ie_i^* \quad \textnormal{(since $e_i = e_iw$)}
	\end{align*}
	for all $1\le i\le n$, and so, we have 
	$$\sum _{e\in s^{-1}(v)}T_e\ast T_{e^*}	=\sum _{i=1}^mT_{e_i}\ast T_{e^*_i}=
	\sum _{i=1}^me_ie_i^*=v=Q_v,$$ thus showing the claim. Then, by the Universal Property of $L_K(E)$, there exists a $K$-algebra homomorphism $\theta_P: L_K(E)\longrightarrow L_K(E)^{\phi_P}$, which maps $u\longmapsto Q_u$, $e\longmapsto T_e$ and $e^*\longmapsto T_{e^*}$.  It is obvious that $Q_u$ and $T_e$ have degree $0$ and $1$ respectively for all $u\in E^0$ and $e\in E^1$. Since $q^{(-1)}_{ij}\in L_K(E)_0$ for all $1\le i, j\le n$, $T_{e^*}$ has degree $-1$ for all $e\in E^1$. This implies that $\phi_P$ is a $\mathbb{Z}$-graded homomorphism,  whence the injectivity of $\theta_P$ is guaranteed by \cite[Theorem 4.8]{tomf:utaisflpa}, thus finishing the proof.
\end{proof}

\noindent The remainder of this section is to investigate Zhang twists $L_K(R_n)^{\lambda}$ of Leavitt path algebras $L_K(R_n)$ by their graded automorphisms $\lambda$ where $R_n$ is the rose graph with $n$ petals. We first note that for any $\lambda \in Aut^{gr}(L_K(R_n))$, by Corollary \ref{gr-IsoofLA}, there exists a unique pair $(P, Q)$  consisting of elements $P$ and $Q$ of $GL_n(L_K(R_n)_0)$ such that $P^{-1} = \phi_P(Q)$,  $\lambda = \phi_P$ and $\lambda^{-1}= \phi_Q$. In light of this note and for convenience, we denote $$L_K(R_n)^{P, Q} :=L_K(R_n)^{\phi_P}= L_K(R_n)^{\lambda}$$ for any such pair $(P, Q)$. As a corollary of Proposition \ref{ZhangTw-SubalLPA}, we obtain that $L_K(R_n)$ is a $K$-subalgebra of all Zhang's twists $L_K(R_n)^{P, Q}$.

\begin{cor}\label{ZhangTw-subal}
	Let $n\ge 2$ be a positive integer, $K$ a field and $R_n$ the rose graph with $n$ petals. Let $P=\left(p_{ij}\right)$ and $Q=\left(q_{ij}\right)$ be elements of $GL_n(L_K(R_n)_0)$ with $P\varphi_P\left(Q\right) =I_n$ and $Q^{-1}=\left(q^{(-1)}_{ij}\right)$. Then, there exists a graded injective homomorphism $\theta_P: L_K(R_n) \longrightarrow   L_K(R_n)^{P, Q}$ of $K$-algebras satisfying
	$$ \theta_P(v)=v, \quad  \theta_P(e_i)=e_i \quad and \quad  \theta_P(e_i^*)=\sum_{k=1}^{n}q^{(-1)}_{ik}e_k^*$$ for all $1\leq i \leq n$.	
\end{cor}
\begin{proof}
	It immediately follows from Proposition \ref{ZhangTw-SubalLPA}.	
\end{proof}

Next we give criteria for the homomorphism $\theta_P$ in Corollary \ref{ZhangTw-subal} to be an isomorphism. In order to do so, we need the following useful fact.

\begin{lem} \label{e_i} 
	Let $n\ge 2$ be a positive integer, $K$ a field and $R_n$ the rose graph with $n$ petals. Let $P=\left(p_{ij}\right)$ and $Q=\left(q_{ij}\right)$ be elements of $GL_n(L_K(R_n)_0)$ with $P\varphi_P\left(Q\right) =I_n$.  For a positive integer $m$, let $P_m=\left(p^{(m)}_{ij}\right)$, $P^{-1}_m=\left(p^{(-m)}_{ij}\right)$, $Q_m=\left(q^{(m)}_{ij}\right)$ and $Q^{-1}_m=\left(q^{(-m)}_{ij}\right)$. Then, the following statements hold:
	
	\begin{itemize}
		\item [(1)] $e_i=\varphi_P^{m}\left(\displaystyle \sum_{k=1}^{n}e_kq_{ki}^{(m)}\right),$
		\item [(2)] $e_i^*=\varphi_P^{m}\left(\displaystyle \sum_{k=1}^{n}q_{ik}^{(-m)}e_k^*\right),$
		\item [(3)] $e_i^*=\varphi_P^{-m}\left(\displaystyle \sum_{k=1}^{n}p_{ik}^{(-m)}e_k^*\right),$	
	\end{itemize}
	for all $1\leq i \leq n$ and $m\geq 1$.
\end{lem}
\begin{proof} We first note that since $P\varphi_P\left(Q\right) =I_n$ and by Proposition \ref{gr-IsoofLA} (2), we obtain that $\phi_{P_m}^{-1} = \phi_{Q_m}$,  $P_m = \phi_{P_m}(Q^{-1}_m)$ and $P^{-1}_m = \phi_{P_m}(Q_m)$ for all $m\ge 1$. Consequently, $\phi_{Q_m}(P^{-1}_m) = \phi_{Q_m}(\phi_{P_m}(Q_m))= \phi_{P_m}^{-1}(\phi_{P_m}(Q_m)) = Q_m$ for all $m\ge 1$. Then, for all $1\leq i \leq n$ and $m\geq 1$, we have
	
	\begin{align*}
	\varphi_P^{m}\left( \sum_{k=1}^{n}e_k q_{ki}^{(m)}\right)&=\varphi_{P_m}\left( \sum_{k=1}^{n}e_k q_{ki}^{(m)}\right)=\sum_{k=1}^{n}\varphi_{P_m}\left(e_k\right)\varphi_{P_m}\left(q_{ki}^{(m)}\right) \\
	&=\sum_{k=1}^{n}\left(\sum_{t=1}^{n}e_t p_{tk}^{(m)}\right) p_{ki}^{(-m)} \quad \textnormal{(since $\varphi_{P_m}(Q_m)=P^{-1}_m$)}\\
	&=\sum_{t=1}^{n}e_t\left(\sum_{k=1}^{n}p_{tk}^{(m)} p_{ki}^{(-m)}\right)=e_i\left(\sum_{k=1}^{n}p_{ik}^{(m)} p_{ki}^{(-m)}\right)\\
	&=e_i v = e_i,
	\end{align*} 
	and 
	\begin{align*}
	\varphi_P^{m}\left(\displaystyle \sum_{k=1}^{n}q_{ik}^{(-m)}e_k^*\right)&=\varphi_{P_m}\left( \sum_{k=1}^{n}q_{ik}^{(-m)}e_k^*\right) = \sum_{k=1}^{n}\varphi_{P_m}\left(q_{ik}^{(-m)}\right)\varphi_{P_m}\left(e_k^*\right)\\
	&=\sum_{k=1}^{n}p_{ik}^{(m)}\left(\sum_{t=1}^{n}p_{kt}^{(-m)}e_t^*\right) \quad \textnormal{(since $\varphi_{P_m}(Q^{-1}_m)=P_m$)}\\
	&= \sum_{t=1}^{n}\left(\sum_{k=1}^{n}p_{ik}^{(m)} p_{kt}^{(-m)}\right) e_t^*=\left( \sum_{k=1}^{n}p_{ik}^{(m)} p_{ki}^{(-m)}\right) e_i^*\\
	&=v e_i^*=e_i^*,
	\end{align*}	
	and
	\begin{align*}
	\varphi_P^{-m}\left(\sum_{k=1}^{n}p_{ik}^{(-m)}e_k^*\right)&=\varphi_Q^{m}\left(\sum_{k=1}^{n}p_{ik}^{(-m)}e_k^*\right)=\varphi_{Q_m}\left(\sum_{k=1}^{n}p_{ik}^{(-m)}e_k^*\right)\\
	&=\sum_{k=1}^{n}\varphi_{Q_m}\left(p_{ik}^{(-m)}\right)\varphi_{Q_m}\left(e_k^*\right)\\
	&=\sum_{k=1}^{n}q_{ik}^{(m)}\left(\sum_{t=1}^{n}q_{kt}^{(-m)}e_t^*\right) \quad \textnormal{(since $\varphi_{Q_m}(P^{-1}_m)=Q_m$)}\\ &= \sum_{t=1}^{n}\left(\sum_{k=1}^{n}q_{ik}^{(m)} q_{kt}^{(-m)}\right) e_t^*=\left( \sum_{k=1}^{n}q_{ik}^{(m)} q_{ki}^{(-m)}\right) e_i^*\\ &=v e_i^*=e_i^*,
	\end{align*}
	thus proving items (1), (2) and (3). This completes the proof of the lemma.
\end{proof}


We are now in a position to characterize when is $L_K(R_n)$ rigid to Zhang twist in the sense that its twist by graded automorphism developed in the previous section turns out to be isomorphic to $L_K(R_n)$. 

\begin{thm}\label{ZhTwisotoLeAl} 
Let $n\ge 2$ be a positive integer, $K$ a field and $R_n$ the rose graph with $n$ petals. Let $P=\left(p_{ij}\right)$ and $Q=\left(q_{ij}\right)$ be elements of $GL_n(L_K(R_n)_0)$ with $P\varphi_P\left(Q\right) =I_n$.  For a positive integer $m$, let $P_m=\left(p^{(m)}_{ij}\right)$, $P^{-1}_m=\left(p^{(-m)}_{ij}\right)$, $Q_m=\left(q^{(m)}_{ij}\right)$ and $Q^{-1}_m=\left(q^{(-m)}_{ij}\right)$. Then, the $K$-algebra homomorphism $\theta_P: L_K(R_n) \longrightarrow   L_K(R_n)^{P, Q}$, defined in Corollary \ref{ZhangTw-subal}, is an isomorphism if and only if $p_{ij}^{(-m)}$, $q_{ij}^{(m)}$, $q_{ij}^{(-m)}$ $\in \textnormal{Im}(\theta_P)$ for all $m\geq 1$ and $1\leq i, j \leq n$. 
\end{thm}
\begin{proof}
	$(\Longrightarrow)$ It is obvious.
	
	$(\Longleftarrow)$ By Corollary \ref{ZhangTw-subal}, $\theta_P$ is always injective, and so it suffices to show that $\theta_P$ is surjective. We first claim that $\alpha$ and $\alpha^*\in\textnormal{Im}(\theta_P)$ for all $\alpha\in (R_n)^*$. We use induction on $|\alpha|$ to establish the claim. If $|\alpha|=1$, then since $e_i = \theta_P (e_i)\in \textnormal{Im}(\theta_P)$ for all $1\le i \le n$, $\alpha \in \textnormal{Im}(\theta_P)$. Since $\theta_P(e_i^*)=\displaystyle \sum_{k=1}^{n}q^{(-1)}_{ik}e_k^*$ for all $1\le i\le n$, we have 
	\begin{align*}
	\begin{pmatrix}
	\theta_P(e_1^*)\\
	\vdots\\
	\theta_P(e_n^*)
	\end{pmatrix}= \begin{pmatrix}
	\displaystyle\sum_{k=1}^{n}q^{(-1)}_{1k}e_k^*\\
	\vdots\\
	\displaystyle\sum_{k=1}^{n}q^{(-1)}_{nk}e_k^*
	\end{pmatrix}= Q^{-1}\begin{pmatrix}
	e_1^*\\
	\vdots\\
	e_n^*
	\end{pmatrix},
	\end{align*}
	and so
	\begin{align*}
	\begin{pmatrix}
	e_1^*\\
	\vdots\\
	e_n^*
	\end{pmatrix}=Q\begin{pmatrix}
	\theta_P(e_1^*)\\
	\vdots\\
	\theta_P(e_n^*)
	\end{pmatrix}.
	\end{align*}
	This follows that $e_i^*=\displaystyle \sum_{k=1}^{n} q_{ik}  \theta_P(e_k^*) = \displaystyle \sum_{k=1}^{n} q_{ik}\ast  \theta_P(e_k^*)$ for all $1\le i\le n$ (since $q_{ij}\in L_K(R_n)_0$ for all $1\le i, j\le n$). By our hypothesis, $q_{ij}\in \textnormal{Im}(\theta_P)$ for all $1\le i, j\le n$, and so $e^*_i = \displaystyle \sum_{k=1}^{n} q_{ik}\ast  \theta_P(e_k^*)\in \textnormal{Im}(\theta_P)$ for all $1\le i\le n$, that means, $\alpha^*\in \textnormal{Im}(\theta_P)$.
	
	Now we proceed inductively, that means, we have $\alpha$ and $\alpha^*\in\textnormal{Im}(\theta_P)$ for all $\alpha\in (R_n)^*$ with $1 < |\alpha|\le m$. For $\alpha\in (R_n)^*$ with $|\alpha|\ge m+1$, we write $\alpha = \beta e_{i_0}$ for some $\beta\in(R_n)^*$ with $|\beta|=m$ and for some $1\le i_0\le n$. By the induction hypothesis, $\beta \in\textnormal{Im}(\theta_P)$. By Lemma \ref{e_i} (1), we have $e_{i_0}= \varphi_P^{m}\left(\displaystyle \sum_{k=1}^{n}e_kq_{ki_{0}}^{(m)}\right)$, 	and so
	$$\alpha=\beta e_i=\beta \varphi_P^{m}(\sum_{k=1}^{n}e_kq_{ki_0}^{(m)})=\beta \ast (\sum_{k=1}^{n}e_k q_{ki_0}^{(m)}).$$
	On the other hand, since $\phi^{-1}_Q = \phi_P$, we have 
	$$Q_m=Q\varphi_Q(Q)\cdots\varphi^{m-1}_Q(Q)=\varphi_P(\varphi_Q(Q)\varphi^2_Q(Q)\cdots\varphi^{m}_Q(Q))=\varphi_P(Q^{-1}Q_{m+1}),$$ 
	so $q_{ki_0}^{(m)}=\varphi_P(\displaystyle\sum_{t=1}^{n}q_{kt}^{(-1)}q_{ti_0}^{(m+1)})$. This shows that
	\begin{align*}
	\alpha &= \beta \ast(\sum_{k=1}^{n}e_k\varphi_P(\displaystyle\sum_{t=1}^{n}q_{kt}^{(-1)}q_{ti_0}^{(m+1)})) = \beta\ast(\sum_{k=1}^{n}( e_k \ast \displaystyle\sum_{t=1}^{n}q_{kt}^{(-1)}q_{ti_0}^{(m+1)}))\\
	&= \beta \ast (\sum_{k=1}^{n} ( e_k \ast (\displaystyle\sum_{t=1}^{n}q_{kt}^{(-1)}*q_{ti_0}^{(m+1)}))) \in  \textnormal{Im}(\theta_P)\ (\text{by our hypothesis}).
	\end{align*}
	
	Write $\alpha^* = \gamma^*e^*_{t_0}$ for some $1\le t_0\le n$  and $\gamma\in (R_n)^*$ with $|\gamma| = m$. By the induction hypothesis, $\gamma^* \in\textnormal{Im}(\theta_P)$. By Lemma \ref{e_i} (3), we have that $e_{t_0}^*=\varphi_P^{-m}\left(\displaystyle \sum_{k=1}^{n}p_{t_0k}^{(-m)}e_k^*\right)$, and hence
	\begin{align*}
	\alpha^*=\gamma^* e_{t_0}^*&=\gamma^*\varphi_P^{-m}(\displaystyle \sum_{k=1}^{n}p_{t_0k}^{(-m)}e_k^*) = \gamma^* \ast (\displaystyle \sum_{k=1}^{n}p_{t_0k}^{(-m)} e_k^*)\\
	&=\gamma^* \ast (\displaystyle \sum_{k=1}^{n}p_{t_0k}^{(-m)} * e_k^*) \in \textnormal{Im}(\theta_P)\ (\text{by our hypothesis}),
	\end{align*}
	thus showing the claim.
	
	We next prove that $\alpha\beta^*\in \textnormal{Im}(\theta_P)$ for all $\alpha$ and $\beta\in (R_n)^*$ with $m:=|\alpha|\ge 1$ and $s:=|\beta|\ge 1$. We use induction on $|\beta|$ to establish the fact. If $|\beta| =1$, then by the above claim, $\alpha$ and $e^*_k\in \textnormal{Im}(\theta_P)$ for all $1\le k\le n$, and so
	\begin{align*}
	\alpha\beta^*&=\alpha e_i^* = \alpha \varphi_P^{m}(\displaystyle \sum_{k=1}^{n}q_{ik}^{(-m)}e_k^*) \quad \textnormal{(by Lemma \ref{e_i} (2))}\\
	&=\alpha \ast (\displaystyle \sum_{k=1}^{n}q_{ik}^{(-m)}e_k^*) = \alpha \ast (\displaystyle \sum_{k=1}^{n}q_{ik}^{(-m)} * e_k^*) \in \textnormal{Im}(\theta_P)\ (\text{by our hypothesis}).
	\end{align*}
	Now we proceed inductively. We need to show that $\alpha\beta^*e_i^* \in \textnormal{Im}(\theta_P)$ for all $1\le i\le n$. We should note that by the induction hypothesis, $\alpha\beta^*\in \textnormal{Im}(\theta_P)$. If $m-s=0$, we have 
	$$\alpha\beta^*e_i^*=\alpha\beta^* * e_i^* \in \textnormal{Im}\left(\theta_P\right).$$
	If $m-s>0$, then we obtain that
	\begin{align*}
	\alpha\beta^*e_i^*&=\alpha\beta^* \varphi_P^{m-s}(\displaystyle \sum_{k=1}^{n}q_{ik}^{(-m+s)}e_k^*)=\alpha\beta^* \ast (\displaystyle \sum_{k=1}^{n}q_{ik}^{(-m+s)}e_k^*)\\
	&=\alpha\beta^* \ast (\displaystyle \sum_{k=1}^{n}q_{ik}^{(-m+s)}*e_k^*) \in \textnormal{Im}(\theta_P).
	\end{align*}
	If $m-s<0$, then we receive that
	\begin{align*}
	\alpha\beta^*e_i^*&=\alpha\beta^* \varphi_P^{m-s}(\displaystyle \sum_{k=1}^{n}p_{ik}^{(m-s)}e_k^*)=\alpha\beta^* \ast (\displaystyle \sum_{k=1}^{n}p_{ik}^{(m-s)}e_k^*\\
	&=\alpha\beta^* \ast (\displaystyle \sum_{k=1}^{n}p_{ik}^{(m-s)}*e_k^*) \in \textnormal{Im}(\theta_P),
	\end{align*}
	proving	the fact. From these observations, we immediately get that $\alpha\beta^*\in \textnormal{Im}(\theta_P)$ for all $\alpha$ and $\beta\in (R_n)^*$. It is obvious that $L_K(R_n)^{P, Q}$ is spanned as a $K$-vector space by $\{\alpha\beta^*\mid \alpha, \beta\in (R_n)^*\}$. This implies that $\textnormal{Im}(\theta_P) = L_K(R_n)^{P, Q}$, that means, $\theta_P$ is surjective, thus finishing the proof.
\end{proof}

Consequently, we provide a simpler criterion for the homomorphism $\theta_P$ be to an isomorphism in the case when $\phi_P(P) = P$.

\begin{cor}\label{varphiP=P}
	Let $n\ge 2$ be a positive integer, $K$ a field and $R_n$ the rose graph with $n$ petals. Let $P= (p_{i,j})$ be an element of $GL_n\left(L_K(R_n)_0\right)$ with $\varphi_P\left(P\right)=P$ and $P^{-1} = (p^{(-1)}_{ij})$. Then, the $K$-algebra homomorphism $\theta_P: L_K(R_n) \longrightarrow L_K(R_n)^{P, P^{-1}}$, defined by 
	\begin{center}
		$\theta_P(v)=v,\ \theta_P(e_i)=e_i \ \ and \ \ \theta_P(e_i^*)=\sum_{k=1}^{n}p_{ik}e_k^*$ \quad for all $1\leq i \leq n$,	
	\end{center}
	is an isomorphism if and only if $p_{ij}$, $p_{ij}^{(-1)} \in \textnormal{Im}(\theta_P)$ for all $1\leq i, j \leq n$. 
\end{cor}
\begin{proof}
	$(\Longrightarrow)$ It is obvious.
	
	$(\Longleftarrow)$ Since $\phi_P(P) = P$ and by Proposition \ref{gr-IsoofLA} (2), $\varphi_P$ is a graded automorphism of $L_K(R_n)$ such that $\varphi_P\left(P^{-1}\right)=P^{-1}$ and $\phi^m_P = \phi_{P^m}$ for all integer $m$. This implies that
	$$\varphi^m_P(P)=P\quad\text{ and }\quad \varphi^m_{P^{-1}}(P^{-1})=P^{-1}$$ for all $m\ge 0$, and so 
	$$P_m  = P\varphi_P\left(P\right)\cdots\varphi^{m-1}_P\left(P\right)=P^m$$ and	$$P^{-1}_m =P^{-1}\varphi_{P^{-1}}\left(P^{-1}\right)\cdots\varphi^{m-1}_{P^{-1}}\left(P^{-1}\right)=P^{-m}$$ for all $m\ge 0$. Since $p_{ij}, p_{ij}^{(-1)} \in L_K(R_n)_0$, we must have
	\begin{center}
		$\smash[b]{P^m=\underbrace{P \ast P \ast\cdots \ast P}_\text{$m$ times}}\vphantom{\underbrace{P \ast P \ast\cdots \ast P}_{m}}$\quad and \quad $\smash[b]{P^{-m}=\underbrace{P^{-1} \ast P^{-1} \ast\cdots \ast P^{-1}}_\text{$m$ times}}\vphantom{\underbrace{P^{-1} \ast P^{-1} \ast\cdots \ast P^{-1}}_{m}}$
	\end{center}
	in $M_n(L_K(R_n)^P)$, that means, $P^m$ and $P^{-m}$ are exactly the $m$th powers of $P$ and $P^{-1}$ in $M_n(L_K(R_n)^P)$, respectively. Then, since $p_{ij}, p_{ij}^{(-1)} \in \textnormal{Im}(\theta_P)$ for all $1\leq i, j \leq n$, all entries of both $P^m$ and $P^{-m}$ lie in $\textnormal{Im}(\theta_P)$ for all $m\geq 1$. By Theorem \ref{ZhTwisotoLeAl}, we immediately obtain that $\theta_P$ is an isomorphism, thus finishing the proof.
\end{proof}

The first consequence of Corollary \ref{varphiP=P} is to show that the Zhang twist $L_K(R_n)^P$ is isomorphic to $L_K(R_n)$ for all $P\in GL_n(K)$. 

\begin{cor}\label{gl-varphiP=P}
	Let $n\ge 2$ be a positive integer, $K$ a field and $R_n$ the rose graph with $n$ petals. Then, for every $P\in GL_n(K)$, the $K$-algebra homomorphism $\theta_P: L_K(R_n) \longrightarrow L_K(R_n)^{P, P^{-1}}$, defined in Corollary \ref{varphiP=P}, is an isomorphism. 
\end{cor}
\begin{proof} Let $P$ be an arbitrary element of $GL_n(K)$.
	By Corollary \ref{gl-IsoofLA}, $\phi_P$ is a graded automorphism of $L_K(R_n)$ with 	$\varphi_P(P)=P$. Moreover, it is obvious that all entries of both $P$ and $P^{-1}$ lie in $\textnormal{Im}(\theta_P)$. Then, by Corollary \ref{varphiP=P}, $\theta_P$ is an isomorphism, thus finishing the proof.
\end{proof}

The second consequence of Corollary \ref{varphiP=P} is to show that the Zhang twist of $L_K(R_n)$ by Anick type graded automorphisms mentioned in Corollary \ref{Anick-IsoofLA}  are isomorphic to $L_K(R_n)$. 

\begin{cor}\label{Anick-varphiP=P}
	Let $n\ge 2$ be a positive integer, $K$ a field and $R_n$ the rose graph with $n$ petals. For every $p\in A_{R_n}(e_1, e_2)\cap L_K(R_n)_0$, consider $U_p=\begin{pmatrix}
	1 & p & 0 & \dots & 0 \\
	0 & 1 & 0 & \dots & 0 \\
	\vdots & \vdots &\vdots &\vdots &\vdots \\
	0 & 0 & 0 & \dots & 1 \\
	\end{pmatrix}\in GL_n(L_K(R_n))_0$. Then the $K$-algebra homomorphism $$\theta_p: L_K(R_n) \longrightarrow L_K(R_n)^{ \phi_{U_p}}$$ defined by $$ \theta_p(v)=v, \quad  \theta_p(e_i)=e_i,\quad \theta_p(e^*_j) = e^*_j \quad and \quad  \theta_p(e_1^*)=e_1^* + pe^*_2$$ for all $1\leq i \leq n$ and $2\le j\le n$, is an isomorphism. 
\end{cor}
\begin{proof} Let $p$ be an arbitrary element of $A_{R_n}(e_1, e_2)\cap L_K(R_n)_0$. By Corollary \ref{Anick-IsoofLA}, $\phi_{U_p}$ is a graded automorphism of $L_K(R_n)$ with 	$ \phi_{U_p}(q)=q$ for all $q\in A_{R_n}(e_1, e_2)\cap L_K(R_n)_0$, where $$U_p=\begin{pmatrix}
	1 & p & 0 & \dots & 0 \\
	0 & 1 & 0 & \dots & 0 \\
	\vdots & \vdots &\vdots &\vdots &\vdots \\
	0 & 0 & 0 & \dots & 1 \\
	\end{pmatrix}\in GL_n(L_K(R_n))_0\quad  \text{ and }\quad U^{-1}_p = U_{-p}.$$ By Theorem \ref{ZhTwisotoLeAl}, $\theta_p := \theta_{U_p}$ is a $K$-algebra homomorphism satisfying  
	$\theta_p(v)=v$,  $\theta_p(e_i)=e_i$, $\theta_p(e^*_j) = e^*_j$  and   $\theta_p(e_1^*)=e_1^* + pe^*_2$ for all $1\leq i \leq n$ and $2\le j\le n$. 
	
	We claim that $\theta_p(p)=p$. Indeed, write $p=\sum \alpha\beta^*$ where $\left| \alpha\right| =\left| \beta\right|=t $ and $\alpha = e_{k_1}e_{k_2}\cdots e_{k_t}$, $\beta^* = e^*_{s_1}e^*_{s_2}\cdots e^*_{s_t}$ with $e_{k_i} \in \left\lbrace e_1, e_3, \ldots e_n\right\rbrace $ and $e^*_{s_i} \in \left\lbrace e^*_2, e^*_3, \ldots, e^*_n\right\rbrace $. Since $ \phi_{U_p}(q)=q$ for all $q\in A_{R_n}(e_1, e_2)\cap L_K(R_n)_0$, we must have $\phi_{U_p}(e_{k_i})=e_{k_i}$ and $\phi_{U_p}(e^*_{s_i})=e^*_{s_i}$ for all $1\le i\le t$. Then, we have that
	\begin{align*}
	\theta_p(\alpha\beta^*)&=\theta_p(\alpha)\ast\theta_p(\beta^*)=\theta_p\left(e_{k_1}e_{k_2}\cdots e_{k_t}\right) \ast \theta_p\left(e^*_{s_1}e^*_{s_2}\cdots e^*_{s_t}\right) \\
	&=\theta_p\left(e_{k_1}\right) \ast \theta_p\left(e_{k_2}\right) \ast \cdots \theta_p\left(e_{k_t}\right) \ast \theta_p\left(e^*_{s_1}\right) \ast \theta_p\left(e^*_{s_2}\right)\ast\cdots\ast \theta_p\left(e^*_{s_t}\right)\\
	&=e_{k_1}\ast e_{k_2}\ast\cdots \ast e_{k_t} \ast e^*_{s_1}\ast e^*_{s_2}\ast\cdots\ast e^*_{s_t}\\
	&=e_{k_1}e_{k_2}\cdots  e_{k_t}  e^*_{s_1}e^*_{s_2}\cdots e^*_{s_t} \quad \textnormal{(since $ \phi_{U_p}(e_{k_i})=e_{k_i}$, $ \phi_{U_p}(e^*_{s_i})=e^*_{s_i}$)}\\
	&=\alpha\beta^*,
	\end{align*}
	and so $p=\theta_p(p)\in \textnormal{Im}\left(\theta_P\right)$. This shows that all entries of both $U_p$ and $U_p^{-1}$ lie in $\textnormal{Im}(\theta_P)$. By Corollary \ref{varphiP=P}, $\theta_p$ is an isomorphism, thus finishing the proof.
\end{proof}

By Corollary \ref{Cuntz-IsoofLA}, for any $u\in U(L_K(R_n)_0)$, there exists a unique graded endomorphism $f_u$ of $L_K(R_n)$ such that $f_u(v) =v$, $f_u(e_i)= ue_i$ and $f_u(e^*_i) = e^*_i u^{-1}$ for all $1\le i\le n$. Moreover, $f_u$ is a graded automorphism if and only if $u^{-1} = f_u(w)$ for some $w\in U(L_K(R_n)_0)$.
In this case, by Corollary \ref{Cuntz-IsoofLA} and Theorem \ref{ZhTwisotoLeAl}, there exists a graded injective homomorphism $$\theta_u:= \theta_{P}: L_K(R_n)\longrightarrow L_K(R_n)^{f_u}$$ of $K$-algebras satisfying $\theta_u(v) = v$, $\theta_u(e_i) = e_i$ and $\theta_u(e^*_i) = e^*_iw^{-1}$ for all $1\le i\le n$, where $P = (e^*_i ue_j)$.	For a positive integer $m$, we always have
$$u_m:=\varphi^{m-1}_u\left(u\right)\cdots f_u\left(u\right)u\in U(L_K(R_n)_0) \text{ and } u^{-1}_m =  u^{-1}f_u(u^{-1})\cdots f^{m-1}_u(u^{-1}).$$ We denote $$P_m := P\phi_P(P)\cdots \phi_P^{m-1}(P)=Pf_u(P)\cdots f_u^{m-1}(P).$$ Moreover, we have the following useful fact.

\begin{lem}\label{P-m}
$P_m=(e_i^*u_me_j)_{1\leq i,j\leq n}$ and $P^{-1}_m=(e_i^*u^{-1}_me_j)_{1\leq i,j\leq n}$ for all $m\ge 1$.
\end{lem}
\begin{proof}
We first claim that 
\begin{center}
	$f_u^m(e_i^*ue_j)=e_i^*u^{-1}_mu_{m+1}e_j$ and $f_w^m(e_i^*we_j)=e_i^*w^{-1}_mw_{m+1}e_j$	
\end{center}
for all $m\geq 1$ and $1\leq i,j\leq n$. We use
induction on $m$ to establish the claim. If $m=1$, then $f_u(e_i^*ue_j)=e_i^*u^{-1}f_u(u)ue_j=e_i^*u^{-1}_1u_{2}e_j$ and $f_w(e_i^*we_j)=e_i^*w^{-1}f_w(w)we_j=e_i^*w^{-1}_1w_{2}e_j$ for all $1\leq i,j\leq n$. 

Now we proceed inductively, that means, we have  $f_u^m(e_i^*ue_j)=e_i^*u^{-1}_mu_{m+1}e_j$ and $f_w^m(e_i^*we_j)=e_i^*w^{-1}_mw_{m+1}e_j$ for all $ 1\leq m\leq k$. For $m=k+1$, we have
\begin{align*}
f_u^m(e_i^*ue_j)&=f_u^{k+1}(e_i^*ue_j)=f_u(f_u^{k}(e_i^*ue_j))\\
& =f_u(e_i^*u^{-1}_ku_{k+1}e_j)=e_i^*u^{-1}f_u(u^{-1}_k)f_u(u_{k+1})ue_j\\
&=e_i^*\left( u^{-1}f_u(u^{-1}_k)\right) \left(f_u(u_{k+1})u\right) e_j\\
&=e_i^*u^{-1}_{k+1}u_{k+2}e_j=e_i^*u^{-1}_{m}u_{m+1}e_j.
\end{align*}
Similarly, we also have $f_w^m(e_i^*we_j)=e_i^*w^{-1}_{m}w_{m+1}e_j$, showing the claim.

We next show the lemma by using induction on $m$. If $m=1$, the statement is obvious. Now we proceed inductively, that means, we have $P_m=(e_i^*u_me_j)_{1\leq i,j\leq n}$ and $P^{-1}_m=(e_i^*u^{-1}_me_j)_{1\leq i,j\leq n}$ for all $1 \leq m\leq k$. For $m=k+1$, we obtain that
$P_m=P_{k+1}=Pf_u(P)\cdots f_u^{k-1}(P)f_u^{k}(P)=P_k f_u^{k}(P)$. By the induction hypothesis, $P_k=(e_i^*u_ke_j)_{1\leq i,j\leq n}$. By the above claim, we have  $$f_u^{k}(P)=\left(f_u^{k}(e_i^*ue_j)\right)_{1\leq i,j\leq n} =\left( e_i^*u^{-1}_ku_{k+1}e_j\right)_{1\leq i,j\leq n}.$$ 
Write $P_m = (p^{(m)}_{i,j})_{1\leq i,j\leq n}$. We then have

\begin{align*}
p^{(m)}_{i,j}&=\sum_{t=1}^n\left( e_i^*u_ke_t\right) \left( e_t^*u_k^{-1}u_{k+1}e_j\right)=e_i^*u_k\left( \sum_{t=1}^n e_t e_t^*\right) u_k^{-1}u_{k+1}e_j\\&=e_i^*u_ku_k^{-1}u_{k+1}e_j=e_i^*u_{k+1}e_j=e_i^*u_{m}e_j
\end{align*}
for all $1\le i, j\le n$, and so $P_m=(e_i^*u_me_j)_{1\leq i,j\leq n}$ and $P^{-1}_m=(e_i^*u^{-1}_me_j)_{1\leq i,j\leq n}$ for all $m\geq 1$, thus finishing the proof.	
\end{proof}

As a corollary of Theorem \ref{ZhTwisotoLeAl}, we obtain a criterion for the Zhang twist $L_K(R_n)^{f_u}$ of $L_K(R_n)$ by a graded automorphism $f_u$ to be isomorphic to $L_K(R_n)$. 

\begin{cor}\label{theta-u}
Let $n\ge 2$ be a positive integer, $K$ a field and $R_n$ the rose with $n$ petals. Let $u$ be an element of $U(L_K(R_n)_0)$ such that  $u^{-1} = f_u(w)$ for some $w\in U(L_K(R_n)_0)$.
Then the following statements hold:
	
$(1)$ The $K$-algebra homomorphism $\theta_u: L_K(R_n) \longrightarrow L_K(R_n)^{f_u}$, defined by $v\longmapsto v$, $e_i\longmapsto e_i$ and $e^*_i\longmapsto e^*_iw^{-1}$ for all $1\le i\le n$, is an isomorphism if and only if $e_i^*u^{-1}_me_j$, $e_i^*w_me_j$,  $e_i^*w^{-1}_me_j \in \textnormal{Im}(\theta_u)$ for all $m\ge 1$ and $1\leq i,j\leq n$.
	
$(2)$ If, in addition, $f_u(u) = u$, then $\theta_u$ is an isomorphism if and only if $e_i^*ue_j, e_i^*u^{-1}e_j\in \textnormal{Im}(\theta_u)$ for all $1\leq i,j\leq n$.
\end{cor}
\begin{proof}
(1) Let $P=(e_i^*ue_j)_{1\leq i,j\leq n}$ and $Q=(e_i^*we_j)_{1\leq i,j\leq n}$. By Lemma \ref{P-m}, we have  $P^{-1}_m=(e_i^*u^{-1}_me_j)_{1\leq i,j\leq n}$,
$Q_m=(e_i^*w_me_j)_{1\leq i,j\leq n}$ and $Q^{-1}_m=(e_i^*w^{-1}_me_j)_{1\leq i,j\leq n}$ for all $m\geq 1$. Then, by Theorem \ref{ZhTwisotoLeAl}, $\theta_u$ is an isomorphism if and only if $e_i^*u^{-1}_me_j$, $e_i^*w_me_j$,  $e_i^*w^{-1}_me_j \in \textnormal{Im}(\theta_u)$ for all $m\ge 1$ and $1\leq i,j\leq n$.

(2) Assume that $f_u(u)=u$. We have $f_u(e^*_iue_j)=e^*_iu^{-1}uue_j=e^*_iue_j$ for all $1\leq i,j\leq n$, and so $f_u(P)=P$. By Corollary \ref{varphiP=P}, $\theta_u$ is an isomorphism if and only if $e_i^*ue_j, e_i^*u^{-1}e_j\in \textnormal{Im}(\theta_u)$ for all $1\leq i,j\leq n$, thus finishing the proof.
\end{proof}

We end this section by presenting the following example which illustrates Corollary \ref{theta-u}. Particularly, we show that  homomorphisms $\theta_u$ are not an isomorphism in general.

\begin{exas}\label{exa-theta}
Let $K$ be a field and $R_2$, the rose graph with $2$ petals.

(1) Let $x = e_1e^*_2 + e_2e^*_1\in L_K(R_n)_0$. It is shown in Examples \ref{Example-auto} (1) that $x$ is a unit of $L_K(R_2)_0$ with $x^{-1} = x$.	We also  have the graded automorphism $f_x$ of $L_K(R_2)$ such that $f_x(v) =v$, $f_x(e_1) = e_2$, $f_x(e_2) = e_1$, $f_x(e_1^*)= e_2^*$ and $f_x(x) = x$, which yields a graded $K$-algebra homomorphism $\theta_x:  L_K(R_2) \longrightarrow L_K(R_2)^{f_x}$ such that $\theta_x(v) =v$, $\theta_x(e_1) = e_1$, $\theta_x(e_2) = e_2$, $\theta_x(e^*_1) = e^*_1x = e^*_2$ and $\theta_x(e^*_2) = e^*_2x = e^*_1$. We then have
\begin{center}
$e^*_1xe_1 =  e^*_2xe_2 = 0 \in \textnormal{Im}(\theta_x)$ and
$e^*_1xe_2 =  e^*_2xe_1 = v \in \textnormal{Im}(\theta_x).$	
\end{center}	
By Corollary \ref{theta-u} (2), we immediately obtain that $\theta_x$ is an isomorphism.\medskip

(2) Let $y = v + e_1^2(e^*_2)^2\in L_K(R_2)_0$. It is shown in Examples \ref{Example-auto} (2) that $y$ is a unit of $L_K(R_2)_0$ with $y^{-1} = v-e_1^2(e^*_2)^2$. We also  have the graded automorphism $f_y$ of $L_K(R_2)$ such that $f_y(v) =v$, $f_y(e_1) = e_1$, $f_y(e_2) = e_2 + e_1^2e^*_2$, $f_y(e_1^*)=e_1^* -e_1(e^*_2)^2$, $f_y(e_2^*)=e_2^*$ and  $f_y(y) = y$, which yields a graded $K$-algebra homomorphism $\theta_y:  L_K(R_2) \longrightarrow L_K(R_2)^{f_y}$ such that $\theta_y(v) =v$, $\theta_y(e_1) = e_1$, $\theta_y(e_2) = e_2$, $\theta_y(e^*_1) = e^*_1y = e^*_1 + e_1(e^*_2)^2$ and $\theta_y(e^*_2) = e^*_2y = e^*_2$. Moreover, we have $e^*_1ye_1 =  e^*_2ye_2 = e^*_1y^{-1}e_1 = e^*_2y^{-1}e_2= v\in \textnormal{Im}(\theta_y)$,
$e^*_2ye_1 =  e^*_2y^{-1}e_1 = 0 \in \textnormal{Im}(\theta_y)$, $e^*_1ye_2 = e_1e^*_2 = \theta_y(e_1e^*_2)\in \textnormal{Im}(\theta_y)$	and $e^*_1y^{-1}e_2 = -e_1e^*_2 = \theta_y(-e_1e^*_2)\in \textnormal{Im}(\theta_y)$. Then, by Corollary \ref{theta-u} (2), we immediately obtain that $\theta_y$ is an isomorphism.

We should note that $\theta_y$ is exactly the automorphism $\theta_{e_1e^*_2}$ of $L_K(R_2)$ introduced in Corollary \ref{Anick-varphiP=P}.\medskip

(3) Let $u = e_1e_2^*+e_2e_1^*+e_1^2e_2^*e_1^*\in L_K(R_2)_0$. 
It is shown in Examples \ref{Example-auto} (3) that $u$ is a unit of $L_K(R_2)_0$ with $u^{-1} = e_1e_2^*+e_2e_1^*-e_2e_1(e_2^*)^2$. We also have the graded automorphism $f_u$ of $L_K(R_2)$ such that $f_u(v) =v$, $f_u(e_1) = e_2+e_1^2e_2^*$, $f_u(e_2) = e_1$ $f_u(e_1^*)=e_2^*$, $f_u(e_2^*)=e_1^*-e_1(e_2^*)^2$
and $f_u(w) = u^{-1}$, where $w =e_1e_2^*+e_2e_1^*-e_2^2e_1^*e_2^*$ and $w^{-1} = e_1e_2^*+e_2e_1^* +e_1e_2(e_1^*)^2$. This yields a graded $K$-algebra homomorphism $\theta_u:  L_K(R_2) \longrightarrow L_K(R_2)^{f_u}$ such that $\theta_u(v) =v$, $\theta_u(e_1) = e_1$, $\theta_u(e_2) = e_2$, $\theta_u(e^*_1) = e^*_1w^{-1} = e^*_2 + e_2(e^*_1)^2$ and $\theta_u(e^*_2) = e^*_2w^{-1} = e^*_1$.\medskip

We claim that $\theta_u$ is not an isomorphism. Indeed, it suffices to show that $e^*_1w_2e_2 \notin \textnormal{Im}(\theta_u)$. We note that $e^*_1w_2e_2$ is exactly the $(1, 2)$-entry of $Q_2 = Q\phi_w(Q)$, where $Q = (e^*_iwe_j)=\begin{pmatrix}
0 & 1\\
1 & -e_2e_1^*
\end{pmatrix}$.

We have $f_w(v)=v$, $f_w(e_1)=we_1=e_2$, $f_w(e_2)=we_2=e_1-e_2^2e_1^*$, $f_w(e_1^*)=e_1^*w^{-1}=e_2^*+e_2e_1^{*2}$ and $f_w(e_2^*)=e_2^*w^{-1}=e_1^*$. We then obtain that
\begin{align*}
f_w(Q)&=f_w\begin{pmatrix}
0 & 1\\
1 & -e_2e_1^*
\end{pmatrix}=\begin{pmatrix}
0 & 1\\
1 & -e_1e_2^*-e_1e_2(e_1^*)^2+e_2^2e_1^*e_2^*
\end{pmatrix},
\end{align*}
and so \begin{align*}
Q_2&=Q f_w(Q)=\begin{pmatrix}
1 & -e_1e_2^*-e_1e_2(e_1^*)^2+e_2^2e_1^*e_2^*\\
-e_2e_1^* & 1+e_2e_2^*+e_2^2(e_1^*)^2
\end{pmatrix}.
\end{align*}
This implies that $$e^*_1w_2e_2 = -e_1e_2^*-e_1e_2(e_1^*)^2+e_2^2e_1^*e_2^*\in L_K(R_2)_0.$$

Assume that $-e_1e_2^*-e_1e_2(e_1^*)^2+e_2^2e_1^*e_2^* \in \textnormal{Im}(\theta_u)$, that means, $-e_1e_2^*-e_1e_2(e_1^*)^2+e_2^2e_1^*e_2^* = \theta_u(a)$ for some $a\in L_K(R_2)_0$. Then, $a$ may be written of the form:
\[a = e_1t_1e^*_1 + e_1t_2e^*_2 + e_2t_3e^*_1 + e_2t_4e^*_2,\]
where $t_i\in L_K(R_2)_0$ for all $1\le i\le 4$. We note that  $f_u(\theta_u(e_1^*))=f_u(e_2^*+e_2(e_1^*)^2)=e_1^*$ and $f_u(\theta_u(e_2^*))=f_u(e_1^*)=e_2^*$ for all $1\leq j\leq 2$. Therefore, we have
\begin{align*}
\theta_u(e_ise_j^*)&=\theta_u(e_i)\ast\theta_u(s)\ast\theta_u(e_j^*)=e_i\ast\theta_u(s)\ast\theta_u(e_j^*)=e_i\phi_u(\theta_u(s))\ast\theta_u(e_j^*)\\
&=e_if_u(\theta_u(s))f_u(\theta_u(e_j^*))=e_if_u(\theta_u(s))e_j^*
\end{align*}
for all $s\in L_K(R_2)_0$ and $1\le i, j\le 2$. This implies that
$$\theta_u(a)=e_1f_u(\theta_u(t_1))e_1^*+e_1f_u(\theta_u(t_2))e_2^*+e_2 f_u(\theta_u(t_3))e_1^*+e_2 f_u(\theta_u(t_4))e_2^*,$$ and so 
\begin{align*}
-e_2e_1^* &= e_1^*(-e_1e_2^*-e_1e_2e_1^{*2}+e_2^2e_1^*e_2^*)e_1=e_1^*\theta_u(a)e_1= f_u(\theta_u(t_1)),\\
-v &=e_1^*(-e_1e_2^*-e_1e_2e_1^{*2}+e_2^2e_1^*e_2^*)e_2=e_1^*\theta_u(a)e_2= f_u(\theta_u(t_2)),\\
0&=e_2^*(-e_1e_2^*-e_1e_2e_1^{*2}+e_2^2e_1^*e_2^*)e_1=e_2^*\theta_u(a)e_1= f_u(\theta_u(t_3)),\\
e_2e_1^*&=e_2^*(-e_1e_2^*-e_1e_2e_1^{*2}+e_2^2e_1^*e_2^*)e_2=e_2^*\theta_u(a)e_2= f_u(\theta_u(t_4)).
\end{align*}
Hence, $e_2e_1^*= f_u(\theta_u(-t_1))= f_u(\theta_u(t_4))$. Since $f_u$ and $\theta_u$ are injective, we must have $-t_1=t_4=:t$, $t_2=-v$ and $t_3=0$. Therefore, $$a=-e_1e_2^*-e_1te_1^*+e_2te_2^*.$$
As above we have shown $ e_2e_1^*=f_u(\theta_u(t))$. On the other hand, 
\begin{align*}
e_2&=f_u(e_1)-e_1^2e_2^*=f_u(e_1)-f_u(e_2^2)f_u(e_1^*)=f_u(e_1-e_2^2e_1^*)\\
e_1^*&=f_u(e_2^*)+e_1e_2^{*2}=f_u(e_2^*)+f_u(e_2)f_u(e_1^{*2})=f_u(e_2^*+e_2e_1^{*2}).
\end{align*}
So $e_2e_1^*=f_u(e_1e_2^*+e_1e_2e_1^{*2}-e_2^2e_1^*e_2^*)=f_u(-\theta_u(a))=f_u(\theta_u(-a))$. It follows that $f_u(\theta_u(t))=f_u(\theta_u(-a))$. By the injectivity of $f_u$ and $\theta_u$, $t=-a$, i.e., $$t=e_1e_2^*+e_1te_1^*-e_2te_2^*,$$
which yields that $(e_1^*)^mte_1^m=t$ for all $m\geq 1$. Since $t\in L_K(R_2)_0$, we write
$t=kv+\sum_{i=1}^{d}k_i\alpha_i\beta_i^*$, where $d\ge 0$, $k, k_i\in K$ and $\alpha_i, \beta_i \in (R_2)^*$ with $|\alpha_i|=|\beta_i|\ge 1$ for all $1\le i\le p$. We have $(e^*_1)^le^l_1 =v$ for all $l\ge 1$ and
\begin{align*}
(e_1^*)^{|\alpha_i|}(\alpha_i\beta^*_i)e_1^{|\alpha_i|}=\begin{cases}
v \quad \textnormal{if }\, \alpha_i=\beta_i^*=e_1^{|\alpha_i|},\\
0 \quad \textnormal{otherwise}\\
\end{cases}
\end{align*}
for all $1\le i\le d$. Hence, for $m = \max\{|\alpha_i|\mid 1\le i\le d\}$, we obtain that

$$t=(e_1^*)^mte_1^m=(e_1^*)^m(kv+\sum_{i=1}^{d}k_i\alpha_i\beta_i^*)e_1^m = c v$$ for some $c\in K$, and so $e^*_1te_2  = e^*_1(cv)e_2= c (e^*_1e_2) =0$.

On the other hand, we have $$e_1^*te_2= e_1^*(e_1e_2^*+e_1te_1^*-e_2te_2^*)e_2 = v,$$ and so $v =0$, a contradiction. Therefore, we must have $-e_1e_2^*-e_1e_2e_1^{*2}+e_2^2e_1^*e_2^*\notin \textnormal{Im}(\theta_u)$, thus $\theta_u$ is not an isomorphism.
\end{exas}

\section{Application: Irreducible representations of $L_K(R_n)$}	
\noindent The study of irreducible representations of Leavitt path algebras is still in its early stage. Chen in his remarkable paper \cite{c:irolpa} initiated the study of simple modules
over Leavitt path algebras. To understand his construction of simple modules, let us first recall some terminologies. Let $E$ be an arbitrary graph. An \textit{infinite path} $p:= e_1\cdots e_n\cdots$ in a graph $E$ is a sequence of edges $e_1, \hdots, e_n, \hdots $ such that $r(e_i) = s(e_{i+1})$ for all $i$. We denote by $E^{\infty}$ the set of all infinite paths in $E$.
For $p:= e_1\cdots e_n\cdots\in E^{\infty}$  and $n\ge 1$, Chen (\cite{c:irolpa}) defines $\tau_{> n}(p) = e_{n+1}e_{n+2}\cdots,$ and $\tau_{\le n}(p) = e_1e_2\cdots e_n$. Two infinite paths $p, q$ are said to be \textit{tail-equivalent} (written $p\sim q$) if there exist positive integers $m, n$ such that $\tau_{> n}(p) = \tau_{> m}(q)$. Clearly $\sim$ is an equivalence relation on $E^{\infty}$, and we let $[p]$ denote the $\sim$ equivalence class of the infinite path $p$.

Let $c$ be a closed path in $E$. Then the path $c c c\cdots$ is an infinite path in $E$, which we denote by $c^{\infty}$. Note that if $c$ and $d$ are closed paths in $E$ such that $c = d^n$, then $c^{\infty}=d^{\infty}$ as elements of $E^{\infty}$. The infinite path $p$ is called \textit{rational} in case $p\sim c^{\infty}$ for some closed path $c$. If $p\in E^{\infty}$ is not rational we say $p$ is \textit{irrational}. We denote by  $E^{\infty}_{irr}$ the set of irrational paths in $E$.

Given a field $K$ and an infinite path $p$, Chen (\cite{c:irolpa}) defines $V_{[p]}$ to be the $K$-vector space having $\{q\in E^{\infty}\mid q\in [p]\}$ as a basis, that is, having basis consisting of distinct elements of $E^{\infty}$ which are tail-equivalent to $p$. $V_{[p]}$ is made a left $L_K(E)$-module by defining, for all $q\in [p]$ and all $v\in E^0$, $e\in E^1$, 

$1)$ $v\cdot q = q$ or $0$ according as $v = s(q)$ or not;

$2)$ $e \cdot q = eq$ or $0$ according as $r(e) = s(q)$ or not;

$3)$ $e^* \cdot q = \tau_1(q)$ or $0$ according as $q = e\tau_1(q)$ or not.\\ 
In \cite[Theorem 3.3]{c:irolpa} Chen showed that $V_{[p]}$ is a simple left $L_K(E)$-module; and  $V_{[p]} \cong V_{[q]}$ if and only if $p \sim q$, which happens precisely when $V_{[p]} = V_{[q]}$. 

Let  $c= e_1\cdots e_t$ be a closed path in $E$ based at $v$ and $f(x) = a_0 + a_1 x + \cdots + a_n x^n$ a polynomial in $K[x]$. We denote by $f(c)$ the element \[f(c):= a_0v + a_1c + \cdots + a_nc^n\in L_K(E).\] 
We denote by $\text{Irr}(K[x])$ the set of all irreducible polynomials in $K[x]$ written in the form $1 - a_1x - \cdots - a_n x^n$ and by $\Pi_c$ the set of all the following closed paths $c_1 := c,\, c_2:= e_2\cdots e_t e_1,\, \ldots, \, c_n := e_n e_1\cdots e_{n-1}$. In \cite[Theorems 4.3 and 4.7]{anhnam} \'{A}nh and the first author proved that for any irreducible polynomial $f\in\text{Irr}(K[x])$, the all cyclic left $L_K(E)$-modules $S^f_{c_i}$ generated by $z$ subject to $z= (a_1c_i + \cdots + a_nc^n_i)z$ ($1\le i\le n$), are both simple and isomorphic to each other, and define a simple $L_K(E)$-module $S^f_{\Pi_c}$. 

In \cite{kn:ataanirolpa} Kuroda and the first author constructed additional classes of simple $L_K(R_n)$-modules  by studying the twisted modules of the simple modules $S^f_{c}$ under Anick
type automorphisms of $L_K(R_n)$ mentioned in Corollary \ref{Anick-IsoofLA}, where $R_n$ is the rose graph with $n$ petals.

For any integer $n\ge 2$, we denote by $C_s(R_n)$ the set of simple closed paths of the form $c = e_{k_1}e_{k_2} \cdots e_{k_m}$, where $k_i\in \{1, 3, \ldots, n\}$ for all $1\le i\le m-1$ and $k_m =2$, in $R_n$. For any $c\in C_s(R_n)$, $p\in A_{R_n}(e_1, e_2)$ and $f\in \text{Irr}(K[x])$, we have a left $L_K(R_n)$-module $S^{f,\, p}_{c}$,  which is the twisted module $(S^f_{c})^{\phi_p}$, where $\phi_p$  is the automorphism of $L_K(R_n)$ defined in Corollary~\ref{Anick-IsoofLA}. By \cite[Theorem 3.6]{kn:ataanirolpa}, the $L_K(R_n)$-module $S^{f,\, p}_{c}$ is always simple.

For each pair $(f, c) \in \text{Irr}(K[x])\times C_s(R_n)$, we define an equivalence relation $\equiv_{f, c}$ on $A_{R_n}(e_1, e_n)$ as follows.
For all $p, q\in A_{R_n}(e_1, e_n)$, $p\equiv_{f, c} q$ if and only if $p-q = rf(c)$ for some $r\in L_K(R_n)$. We denote by $[p]$ the $\equiv_{f, c}$ equivalence class of $p$. The following theorem provides us with a list of some classes of pairwise non-isomorphic simple $L_K(R_n)$-modules.

\begin{thm}[{\cite[Theorem 3.8]{kn:ataanirolpa}}]\label{Irrrep1} 
	Let $K$ be a field, $n\ge 2$ a positive integer, and $R_n$ the rose graph with $n$ petals. Then, the following set 
	$$\{V_{[\alpha]}\mid \alpha\in (R_n)^{\infty}_{irr}\} \sqcup \{S^f_{\Pi_c}\mid c\in SCP(R_n),\, f \in \textnormal{Irr}(K[x])\}\, \sqcup$$
	\[\sqcup\, \{S^{f,\, p}_{d} \mid d\in C_s(R_n),\, f \in \textnormal{Irr}(K[x]),\, [0]\neq [p]\in A_{R_n}(e_1, e_2)/\equiv_{f, d}\}\]
	consists of pairwise non-isomorphic simple left $L_K(R_n)$-modules.	
\end{thm}

The remainder of this section is to investigate the twisted modules $(V_{[\alpha]})^{\phi_P}$ of the simple $L_K(R_n)$-modules $V_{[\alpha]}$ by graded automorphisms $\phi_P$ mentioned in Corollary \ref{gr-IsoofLA}, where $p$ is an infinite path in $R_n$ and $P\in GL_n(K)$. For convenience, we denote  $$V^P_{[\alpha]} := (V_{[\alpha]})^{\phi^{-1}_P} = (V_{[\alpha]})^{\phi_{P^{-1}}}$$ 
for any $\alpha\in (R_n)^{\infty}$ and $P\in GL_n(K)$. Denoting by $\cdot$ the module operation in $V^P_{[\alpha]}$, we have $v\cdot \beta = \phi^{-1}_P(v)\beta = v\beta =\beta$, 

\[e_i \cdot \beta = \phi^{-1}_{P}(e_i)\beta = \phi_{P^{-1}}(e_i)\beta = (\sum^n_{t=1}p'_{ti}e_t)\beta \] and
\[e^*_i \cdot \beta = \phi^{-1}_P(e^*_i)\beta = \phi_{P^{-1}}(e^*_i)\beta = (\sum^n_{t=1}p_{it}e^*_t)\beta \ \ \text{in } V_{[\alpha]}\]
for all $\beta\in [\alpha]$ and $1\le i\le n$, where $P= (p_{ij})$ and $P^{-1} = (p'_{ij})\in GL_n(K)$.

We note that the symmetric group $S_n$ acts on the set $(R_n)^{\infty}$ by setting: $$(\sigma, p = e_{i_1}e_{i_2}\cdots e_{i_m}\cdots)\longmapsto \sigma\cdot p = e_{\sigma(i_1)}e_{\sigma(i_2)}\cdots e_{\sigma(i_m)}\cdots$$ for all $\sigma\in S_n$ and $p= e_{i_1}e_{i_2}\cdots e_{i_m}\cdots\in (R_n)^{\infty}$. The orbit of $p$ is the set $\{\sigma\cdot p\mid \sigma\in S_n\}$ and denoted by
$S_n\cdot p$. The set of orbits of points $p$ in $(R_n)^{\infty}$ under the action of $S_n$ form a partition of $(R_n)^{\infty}$. The associated equivalence relation is defined by saying $p\sim q$ if and only of there exists an element $\sigma\in S_n$ such that $q =\sigma\cdot p$. Moreover, we have that $(R_n)^{\infty}_{irr}$ is an invariant subset of $(R_n)^{\infty}$, that means, $$S_n\cdot (R_n)^{\infty}_{irr} := \{\sigma\cdot p\mid p\in (R_n)^{\infty}_{irr}\} = (R_n)^{\infty}_{irr}.$$ 
We denote by $(R_n)^{\infty}_{irr-eeri}$ the set of all irrational paths $p= e_{i_1}e_{i_2}\cdots e_{i_m}\cdots$ such that each edge is repeated infinitely many times in the path, that is, \[|\{m\in \mathbb{N}\mid e_{i_m} = e_j\}| = \infty\] for all $1\le j\le n$. It is not hard to see that  $(R_n)^{\infty}_{irr-eeri}$ is an invariant subset of $(R_n)^{\infty}$, and $(R_2)^{\infty}_{irr-eeri} = (R_2)^{\infty}_{irr}$.

We also have a group action of $S_n$ on the general linear group $GL_n(K)$ defined by: \[(\sigma, A= [a_1\ a_2\ \cdots\ a_n])\longmapsto \sigma\cdot A := [a_{\sigma(1)}\ a_{\sigma(2)}\ \cdots\ a_{\sigma(n)}]\] for all $\sigma\in S_n$ and $A= [a_1\ a_2\ \cdots\ a_n]\in GL_n(K)$, where $a_j$ is the $j^{\text{th}}$ column of $A$. In the following theorem, we describe simple $L_K(R_n)$-modules $V^P_{[\alpha]}$ associated to pairs $(\alpha, P)\in (R_n)^{\infty}_{irr-eeri}\times GL_n(K)$.

\begin{thm}\label{Irrrep2-irr} 
	Let $K$ be a field, $n\ge 2$ a positive integer, and $R_n$ be the rose graph with $n$ petals. Let $P = (p_{ij})\in GL_n(K)$ be an arbitrary element and $\alpha=e_{i_1}e_{i_2}\cdots e_{i_m}\cdots\in  (R_n)^{\infty}_{irr-eeri}$. Then, the following statements hold:
	
	$(1)$ $V^P_{[\alpha]}$ is a simple left $L_K(R_n)$-module;
	
	$(2)$ $End_{L_K(R_n)}(V^P_{[\alpha]})\cong K$;
	
	$(3)$ $V^P_{[\alpha]}\cong L_K(R_n)/\bigoplus^{\infty}_{m=0}L_K(R_n)(\phi_P(\epsilon_m) - \phi_P(\epsilon_{m+1}))$, where $\epsilon_0 := v$, $\epsilon_m = e_{i_1}\cdots e_{i_m}e^*_{i_m}\cdots e^*_{i_1}$ for all $m\ge 1$, and the graded automorphism $\phi_P$ is defined in Corollary \ref{gr-IsoofLA}. Consequently, $V^P_{[\alpha]}$ is not finitely presented.
	
	$(4)$ For any $\beta\in (R_n)^{\infty}_{irr-eeri}$, $V_{[\beta]}\cong V^P_{[\alpha]}$ if and only if there exist an element $\sigma\in S_n$ and a diagonal matrix $D\in GL_n(K)$ such that $P = \sigma\cdot D$ and $\sigma\cdot \beta \sim \alpha$.
	
	$(5)$ For any $\beta\in (R_n)^{\infty}_{irr-eeri}$ and any $Q\in GL_n(K)$, $V^Q_{[\beta]}\cong V^P_{[\alpha]}$ if and only if there exist an element $\sigma\in S_n$ and a diagonal matrix $D\in GL_n(K)$ such that $Q^{-1}P = \sigma\cdot D$ and $\sigma\cdot\beta \sim \alpha $.
\end{thm}
\begin{proof}
	(1) It follows from the fact that $V_{[\alpha]}$ is a simple left $L_K(R_n)$-module (by \cite[Theorem 3.3 (1)]{c:irolpa}) and $\phi_{P^{-1}}$ is an automorphism of $L_K(R_n)$ (by Corollary \ref{gr-IsoofLA}).
	
	(2) By \cite[Theorem 3.3 (1)]{c:irolpa}, we have $End_{L_K(R_n)}(V_{[\alpha]})\cong K$, which yields that $End_{L_K(R_n)}(V^P_{[\alpha]})\cong K$.
	
	(3) Since $V_{[\alpha]}$ is a simple left $L_K(R_n)$-module, $V_{[\alpha]} = L_K(R_n)\alpha$. By \cite[Theorem 3.4]{anhnam}, we obtain that $$\{r\in L_K(R_n)\mid r\alpha =0 \text{ in } V_{[\alpha]}\} =\bigoplus^{\infty}_{m=0}L_K(R_n)(\epsilon_m - \epsilon_{m+1}),$$ where $\epsilon_0 := v$ and $\epsilon_m = e_{i_1}\cdots e_{i_m}e^*_{i_m}\cdots e^*_{i_1}\in L_K(R_n)$ for all $m\ge 1$. By item (1), $V^P_{[\alpha]}$ is a simple left $L_K(R_n)$-module, and so $V^P_{[\alpha]} = L_K(R_n)\cdot \alpha$, that means, every element of $V^P_{[\alpha]}$ is of the form $r\cdot \alpha = \phi_{P^{-1}}(r)\alpha$, where $r\in L_K(R_n)$. We next compute $\text{ann}_{L_K(R_n)}(\alpha):=\{r\in L_K(R_n)\mid r\cdot\alpha =0\}$. Indeed, let $r\in \text{ann}_{L_K(R_n)}(\alpha)$. We then have $\phi_{P^{-1}}(r)\alpha = r\cdot \alpha = 0$ in $V_{[\alpha]}$, which gives that $\phi_{P^{-1}}(r) = \sum^{k}_{i=1}r_i(\epsilon_{m_i} - \epsilon_{m_i+1})$, where $k\ge 1$ and $r_i\in L_K(R_n)$ for all $1\le i\le k$, and so $$r = \phi_P(\phi_{P^{-1}}(r))=\sum^{k}_{i=1}\phi_P(r_i)\left(\phi_P(\epsilon_{m_i}) - \phi_P(\epsilon_{m_i+1})\right).$$ This implies that $$\text{ann}_{L_K(R_n)}(\alpha)\subseteq \bigoplus^{\infty}_{m=0}L_K(R_n)(\phi_P(\epsilon_m) - \phi_P(\epsilon_{m+1})).$$ Conversely, assume that $r\in \bigoplus^{\infty}_{m=0}L_K(R_n)(\phi_P(\epsilon_m) - \phi_P(\epsilon_{m+1}))$; i.e., $r = \sum^{k}_{i=1}r_i(\phi_P(\epsilon_{m_i}) - \phi_P(\epsilon_{m_i+1}))$, where $k\ge 1$ and $r_i\in L_K(R_n)$ for all $1\le i\le k$. We then have 
	\[r\cdot \alpha = \phi_{P^{-1}}(r)\alpha = (\sum^{k}_{i=1} \phi_{P^{-1}}(r_i)\left(\epsilon_{m_i} - \epsilon_{m_i+1})\right)\alpha=0\] in $V_{[\alpha]}$, and so $r \in \text{ann}_{L_K(R_n)}(\alpha)$, showing that $$\bigoplus^{\infty}_{m=0}L_K(R_n)(\phi_P(\epsilon_m) - \phi_P(\epsilon_{m+1}))\subseteq \text{ann}_{L_K(R_n)}(\alpha).$$ Hence $\bigoplus^{\infty}_{m=0}L_K(R_n)(\phi_P(\epsilon_m) - \phi_P(\epsilon_{m+1}))= \text{ann}_{L_K(R_n)}(\alpha).$ This implies that $$V^P_{[\alpha]}\cong L_K(R_n)/\bigoplus^{\infty}_{m=0}L_K(R_n)(\phi_P(\epsilon_m) - \phi_P(\epsilon_{m+1})).$$ 
	
	Assume that $V^P_{[\alpha]}$ is finitely presented. This shows that $\bigoplus^{\infty}_{m=0}L_K(R_n)(\phi_P(\epsilon_m) - \phi_P(\epsilon_{m+1}))$ is finitely
	generated, whence there exists an integer $k\ge 1$ such that 
	$\phi_P(\epsilon_m) = \phi_P(\epsilon_{m+k})$ for all $m\ge 0$; equivalently, $\epsilon_m = \epsilon_{m+k}$ for all $m\ge 0$ (since $\phi_P$ is an automorphism), but this cannot happen in $L_K(R_n)$. Therefore, $V^P_{[\alpha]}$ is not finitely presented.
	
	(4) $(\Leftarrow)$ Assume that there exist an element $\sigma\in S_n$ and a diagonal matrix $D\in GL_n(K)$ such that $P = \sigma\cdot D$ and $\sigma\cdot \alpha \sim \beta$. We then have $\sigma\cdot \alpha = e_{\sigma(i_1)}e_{\sigma(i_2)}\cdots e_{\sigma(i_m)}\cdots\in  (R_n)^{\infty}$ and $V_{[\beta]} \cong V_{[\sigma\cdot \alpha]}$ (by Theorem \ref{Irrrep1}). By \cite[Theorem 3.4]{anhnam}, $V_{[\sigma\cdot \alpha]} \cong L_K(R_n)/\bigoplus^{\infty}_{m=0}L_K(R_n)(\lambda_m - \lambda_{m+1})$, where $\lambda_0 = v$ and $\lambda_m = e_{\sigma(i_1)}\cdots e_{\sigma(i_m)}e^*_{\sigma(i_m)}\cdots e^*_{\sigma(i_1)}$ for all $m\ge 1$.
	
	On the other hand, by Item (3), 	$V^P_{[\alpha]}\cong L_K(R_n)/\bigoplus^{\infty}_{m=0}L_K(R_n)(\phi_P(\epsilon_m) - \phi_P(\epsilon_{m+1}))$, where $\epsilon_0 := v$, $\epsilon_m = e_{i_1}\cdots e_{i_m}e^*_{i_m}\cdots e^*_{i_1}$ for all $m\ge 1$.
	Write $P = (p_{ij})$ and $P^{-1} = (q_{ij})$. Then, since $P= \sigma\cdot D$, we have $p_{i\sigma(i)} \neq 0$ and $p_{ij}=0$ for all $1\le i, j\le n$ and $j \neq \sigma(i)$. This implies that $q_{\sigma(i)i}= p_{i\sigma(i)}^{-1}$ and $q_{ki}= 0$ for all $1\le i, k\le n$ and $k \neq \sigma(i)$, and so 
	\begin{center}
		$\phi_P(e_i) = \sum^n_{k=1}p_{ki}e_k= p_{k\sigma(k)}e_k$ and $\phi_P(e^*_i) = \sum^n_{k=1}q_{ik}e^*_k= q_{\sigma(k)k}e^*_k$ 
	\end{center}
	for all $1\le i\le n$, where $i= \sigma(k)$. This shows that \begin{align*}
	\varphi_P(\epsilon_m)&=\varphi_P\left(e_{i_1} \cdots e_{i_m} e^*_{i_m} \cdots e^*_{i_1}\right) 
	=\varphi_P\left(e_{i_1}\right) \cdots \varphi_P\left(e_{i_m}\right) \varphi_P\left(e^*_{i_m}\right)\cdots \varphi_P\left(e^*_{i_1}\right)\\
	&=p_{i_1\sigma(i_1)}e_{\sigma(i_1)}\cdots p_{i_m\sigma(i_m)}e_{\sigma(i_m)}q_{\sigma(i_m)i_m}e^*_{\sigma(i_m)}\cdots q_{\sigma(i_1)i_1}e^*_{\sigma(i_1)}\\
	&= e_{\sigma(i_1)}\cdots e_{\sigma(i_m)}e^*_{\sigma(i_m)}\cdots e^*_{\sigma(i_1)} =\lambda_m
	\end{align*}
	for all $m\ge 1$, and so $$V^P_{[\alpha]}\cong L_K(R_n)/\bigoplus^{\infty}_{m=0}L_K(R_n)(\lambda_m - \lambda_{m+1})\cong V_{[\sigma\cdot \alpha]}\cong V_{[\beta]},$$ as desired.
	
	$(\Rightarrow)$ Assume that $\theta: V_{[\beta]}\longrightarrow V^P_{[\alpha]}$ is an isomorphism of left $L_K(R_n)$-modules. Let $q\in [\beta]$ be an element such that $\theta(q) = \sum^m_{i=1}k_i\alpha_i$, where $m$ is minimal such that $k_i\in K\setminus\{0\}$ and all the $\alpha_i$ are pairwise distinct in $[\alpha]$. Write $q=e_{t_1}e_{t_2}\cdots e_{t_k}\cdots\in  (R_n)^{\infty}_{irr-eeri}$ and $\alpha_i = e_{j_{i1}}e_{j_{i2}}\cdots e_{j_{ik}}\cdots\in  (R_n)^{\infty}_{irr-eeri}$, where $1\le t_i, j_{ik}\le n$. By the minimality of $m$, we have $$0\neq \theta(\tau_{>1}(q)) = \theta(e^*_{t_1}q) = e^*_{t_1}\cdot \theta(q) = (\sum^n_{j=1}p_{t_1j}e^*_j)(\sum^m_{i=1}k_i\alpha_i)= \sum^m_{i=1}k^{(1)}_i\tau_{>1}(\alpha_i),$$ where $k^{(1)}_i = k_ip_{t_1j_{i1}}\in K\setminus\{0\}$ for all $1\le i\le m$, and all the $\tau_{>1}(\alpha_i)$ are pairwise distinct in $[\alpha]$. For all $s\neq t_1$, we have $$0 = \theta(e^*_s q) = e^*_s\cdot \theta(q) = (\sum^n_{j=1}p_{sj}e^*_j)(\sum^m_{i=1}k_i\alpha_i)= \sum^m_{i=1}k_ip_{sj_{i1}}\tau_{>1}(\alpha_i).$$
	Since all the $\tau_{>1}(\alpha_i)$ are pairwise distinct, they are linearly independent in $V^P_{[\alpha]}$, and so $k_ip_{sj_{i1}} =0$ for all $1\le i\le m$, this yields $p_{sj_{i1}}=0$ for all $1\le i\le m$ and $s\neq t_1$; that means, for each $1\le i\le n$, the $j_{i1}^{th}$-column of $P$ has only the $(t_1, j_{i1})$-entry is nonzero. Assume that there exist two numbers $1\le i\neq k\le m$ such that $\tau_{\le 1}(\alpha_i) \neq \tau_{\le 1}(\alpha_k)$, i.e, $e_{j_{i1}}\neq e_{j_{k1}}$. We then have $p_{t_1j_{i1}}\neq 0$, $p_{t_1j_{k1}}\neq 0$ and $p_{sj_{i1}}=0 = p_{sj_{k1}}$ for all $s\neq t_1$, and so $A$ is not invertible, a contradiction. This implies that  $\tau_{\le 1}(\alpha_i) = \tau_{\le 1}(\alpha_j)$ for all $1\le i, j\le m$, and the $t_1^{th}$-row of $P$ has only the $(t_1, j_{i1})$-entry is nonzero.
	
	If $e_{t_2}= e_{t_1}$, we then have  $$0\neq \theta(\tau_{>2}(q)) = \theta(e^*_{t_2}\tau_{>1}(q)) = e^*_{t_2}\cdot \theta(\tau_{>1}(q)) =  \sum^m_{i=1}k^{(2)}_i\tau_{>2}(\alpha_i),$$ where $k^{(2)}_i= k^{(1)}_ip_{t_1j_{i1}}\in K\setminus\{0\}$ for all $1\le i\le m$. By the minimality of $m$, all the $\tau_{>2}(\alpha_i)$ are pairwise distinct in $[\alpha]$ and $\tau_{\le 1}(\tau_{>2}(\alpha_i)) = \tau_{\le 1}(\tau_{>2}(\alpha_k))$ for all $1\le i, k\le m$.
	
	If $e_{t_2}\neq  e_{t_1}$, then by using using the quality $$0\neq \theta(\tau_{>1}(q)) =  \sum^m_{i=1}k^{(1)}_i\tau_{>1}(\alpha_i)$$ and repeating the above same argument which was done for $e_{t_1}$, we obtain that the $j^{th}_{i1}$-column and $t^{th}_2$-row of $P$ have only that the $(t_2, j_{i2})$-entry is nonzero, all the $\tau_{>2}(\alpha_i)$ are pairwise distinct in $[\alpha]$ and $\tau_{\le 1}(\tau_{>2}(\alpha_i)) = \tau_{\le 1}(\tau_{>2}(\alpha_k))$ for all $1\le i, k\le m$. Therefore, in any case, we have that all the $\tau_{>2}(\alpha_i)$ are pairwise distinct in $[\alpha]$ and $\tau_{\le 2}(\alpha_i) = \tau_{\le 2}(\alpha_k)$ for all $1\le i, k\le m$.
	
	By repeating this process, we obtain that $\tau_{\leq l}(\alpha_i) = \tau_{\leq l}(\alpha_j)$ for all $l\ge 1$ and $1\le i, j\le m$, and every row and every column of $P$ has only a nonzero entry (since $q\in (R_n)^{\infty}_{irr-eeri}$). Then, since all the $\tau_{\leq l}(\alpha_i)$ are the same for all $l\ge 1$, and all the $\alpha_i$ are pairwise distinct, we must have $m=1$. Since every row and every column of $P$ has only a nonzero entry, there exists an element $\sigma\in S_n$  such that $p_{i\sigma(i)}\neq 0$ for all $1\le i\le n$. This implies that $P = \sigma\cdot D$ for some diagonal matrix $D\in GL_n(K)$ and $\sigma\cdot q = e_{\sigma(t_1)}e_{\sigma(t_2)}\cdots e_{\sigma(t_k)}\cdots = \alpha_1$, this yields $\sigma\cdot q \sim \alpha$. Since $q \sim \beta$, there exists natural numbers $s$ and $l$ such that $\tau_{> s}(q) = \tau_{> l}(\beta)$, and so $$\sigma\cdot\beta\sim\sigma\cdot \tau_{> l}(\beta) = \sigma\cdot \tau_{> s}(q) \sim \sigma\cdot q\sim \alpha,$$ as desired.
	
	(5) We note that 
	\begin{align*}
	V^Q_{[\beta]}\cong V^P_{[\alpha]}&\Longleftrightarrow  (V_{[\alpha]})^{\phi_{P^{-1}}}\cong (V_{[\beta]})^{\phi_{Q^{-1}}}\Longleftrightarrow  (V_{[\beta]})^{\phi_{Q^{-1}}})^{\phi_Q}\cong (V_{[\alpha]})^{\phi_{P^{-1}}})^{\phi_Q}\\
	& \Longleftrightarrow V_{[\beta]}\cong (V_{[\alpha]})^{\phi_{P^{-1}Q}} =V^{Q^{-1}P}_{[\alpha]}.
	\end{align*} 
	Using this note and Item (4), we immediately get the statement, thus finishing the proof.
\end{proof}

For any integer $n\ge 2$, we define an equivalence relation $\equiv$ on $(R_n)^{\infty}_{irr-eeri}$ as follows. For all $\alpha, \beta\in (R_n)^{\infty}_{irr-eeri}$, $\alpha \equiv \beta$ if and only if $\sigma \cdot \alpha \sim \beta$ for some $\sigma\in S_n$. We denote by $[\alpha]_{\equiv}$ the $\equiv$ equivalence class of $\alpha$. The following corollary shows that all simple $L_K(R_n)$-modules $V^P_{[\alpha]}$ associated to pairs $(\alpha, P)\in (R_n)^{\infty}_{irr-eeri}\times GL_n(K)$ may be parameterized by the set $((R_n)^{\infty}_{irr-eeri}/\equiv) \times\ GL_n(K)$.

\begin{cor}\label{Irrrep3-irr}
Let $K$ be a field, $n\ge 2$ a positive integer and $R_n$ be the rose graph with $n$ petals. Then, the set $$\{V^P_{[\alpha]}\mid [\alpha]_{\equiv} \in (R_n)^{\infty}_{irr-eeri}/\equiv \text{ and } P \in GL_n(K)\}$$ consists of pairwise non-isomorphic simple left $L_K(R_n)$-modules.	
\end{cor}
\begin{proof}
Let $\alpha$ and $\beta$ be elements of $(R_n)^{\infty}_{irr-eeri}$ such that $[\alpha]_{\equiv}\neq [\beta]_{\equiv}$. We then have that $\sigma\cdot \alpha$ is not tail-equivalent to $\beta$ for all $\sigma\in S_n$. By Theorem \ref{Irrrep2-irr} (5), $V_{[\alpha]}^P \ncong V_{[\beta]}^Q$ as left $L_K(R_n)$-modules  for all $P, Q\in GL_n(K)$, which yields the statement, thus finishing the proof.
\end{proof}

For any integer $n\ge 2$ and any field $K$, we denote by $\mathbb{U}_n(K)$ the subgroup of $GL_n(K)$ consisting all upper-triangle matrices  with $1$'s along the diagonal. As the second corollary of Theorem \ref{Irrrep2-irr}, we obtain that all  simple $L_K(R_n)$-modules $V^P_{[\alpha]}$ associated to pairs $(\alpha, P)\in (R_n)^{\infty}_{irr-eeri} \times \mathbb{U}_n(K)$ may be parameterized by the set $((R_n)^{\infty}_{irr-eeri}/\sim) \times\ \mathbb{U}_n(K)$.

\begin{cor}\label{Irrrep4-irr}
Let $K$ be a field, $n\ge 2$ a positive integer, $R_n$ the rose graph with $n$ petals and $\mathbb{U}_n(K)$ the subgroup of $GL_n(K)$ consisting all upper-triangle matrices  with $1$'s along the diagonal. Let $\alpha$ and $\beta$ be elements of $(R_n)^{\infty}_{irr-eeri}$ and let $P$ and $Q$ be elements of $\mathbb{U}_n(K)$. Then, $V_{[\alpha]}^P \cong V_{[\beta]}^Q$ if and only if $\alpha\sim \beta$ and $P=Q$. 	Consequently, the set $$\{V^P_{[\alpha]}\mid \alpha \in (R_n)^{\infty}_{irr-eeri} \text{ and } P \in \mathbb{U}_n(K)\}$$ consists of pairwise non-isomorphic simple left $L_K(R_n)$-modules.	
\end{cor}
\begin{proof}
$(\Rightarrow)$ Assume that	$V_{[\alpha]}^P \cong V_{[\beta]}^Q$. Then, by Theorem \ref{Irrrep2-irr} (5), there exist an element $\sigma\in S_n$ and a diagonal matrix $D\in GL_n(K)$ such that $Q^{-1}P = \sigma\cdot D$ and $\sigma\cdot\beta \sim \alpha $. Since $P, Q\in \mathbb{U}_n(K)$, we have $\sigma\cdot D=Q^{-1}P\in \mathbb{U}_n(K)$, and so $\sigma = 1_{S_n}$ and $D = I_n$. This implies that $P= Q$ and $\alpha\sim \beta$.
	
	$(\Leftarrow)$  It immediately follows from Theorem \ref{Irrrep2-irr} (5), thus finishing the proof.
\end{proof}

In the following theorem, we describe simple $L_K(R_n)$-modules $V^P_{[c^{\infty}]}$ associated to pairs $(c, P)\in SCP(R_n)\times GL_n(K)$.

\begin{thm}\label{Irrrep5-irr} 
Let $K$ be a field, $n\ge 2$ a positive integer, and $R_n$ be the rose graph with $n$ petals. Let $P = (p_{ij})\in GL_n(K)$ be an arbitrary element and $c\in  SCP(R_n)$. Then, the following statements hold:
	
$(1)$ $V^P_{[c^{\infty}]}$ is a simple left $L_K(R_n)$-module;
	
$(2)$ $End_{L_K(R_n)}(V^P_{[c^{\infty}]})\cong K$;
	
$(3)$ $V^P_{[c^{\infty}]}\cong L_K(R_n)/L_K(R_n)(v - \phi_P(c))$, where the graded automorphism $\phi_P$ is defined in Corollary \ref{gr-IsoofLA}. 

$(4)$ For any $d\in SCP(R_n)$, $V_{[d^{\infty}]}\cong V^P_{[c^{\infty}]}$ if and only if $d = \phi_P(\beta)$ for some $\beta\in \Pi_c$.
	
$(5)$ For any $d\in SCP(R_n)$ and any $Q\in GL_n(K)$, $V^Q_{[d^{\infty}]}\cong V^P_{[c^{\infty}]}$ if and only if $\phi_Q(d) = \phi_P(\beta)$ for some $\beta\in \Pi_c$.
\end{thm}
\begin{proof}
(1) It follows from the fact that $V_{[c^{\infty}]}$ is a simple left $L_K(R_n)$-module (by \cite[Theorem 3.3 (1)]{c:irolpa}) and $\phi_{P^{-1}}$ is an automorphism of $L_K(R_n)$ (by Corollary \ref{gr-IsoofLA}).

(2) By \cite[Theorem 3.3 (1)]{c:irolpa}, we have $End_{L_K(R_n)}(V_{[c^{\infty}]})\cong K$, which yields that $End_{L_K(R_n)}(V^P_{[c^{\infty}]})\cong K$.

(3) Since $V_{[c^{\infty}]}$ is a simple left $L_K(R_n)$-module, $V_{[c^{\infty}]} = L_K(R_n)c^{\infty}$. By \cite[Theorem 4.3]{anhnam} (see also \cite[Theorem 2.8]{amt:eosmolpa}), we obtain that $$\{r\in L_K(R_n)\mid rc^{\infty} =0 \text{ in } V_{[c^{\infty}]}\} =L_K(R_n)(v - c).$$ By item (1), $V^P_{[c^{\infty}]}$ is a simple left $L_K(R_n)$-module, and so $V^P_{[c^{\infty}]} = L_K(R_n)\cdot c^{\infty}$, that means, every element of $V^P_{[c^{\infty}]}$ is of the form $r\cdot c^{\infty} = \phi_{P^{-1}}(r)c^{\infty}$, where $r\in L_K(R_n)$. We next compute $\text{ann}_{L_K(R_n)}(c^{\infty}):=\{r\in L_K(R_n)\mid r\cdot c^{\infty} =0\}$. Indeed, let $r\in \text{ann}_{L_K(R_n)}(c^{\infty})$. We then have $\phi_{P^{-1}}(r)c^{\infty}= r\cdot c^{\infty} = 0$ in $V_{[c^{\infty}]}$, which gives that $\phi_{P^{-1}}(r) = s(v - c)$ for some $s\in L_K(R_n)$, and so $$r = \phi_P(\phi_{P^{-1}}(r))=\phi_P(s)\left(v - \phi_P(c)\right).$$ This implies that $$\text{ann}_{L_K(R_n)}(c^{\infty})\subseteq L_K(R_n)(v - \phi_P(c)).$$ Conversely, assume that $r\in L_K(R_n)(v - \phi_P(c))$; i.e., $r = x(v - \phi_P(c))$ for some  $x\in L_K(R_n)$. We then have 
\[r\cdot c^{\infty} = \phi_{P^{-1}}(r)c^{\infty} =\phi_{P^{-1}}(x(v - \phi_P(c)))c^{\infty} = \phi_{P^{-1}}(x)\left(v - c\right)c^{\infty}=0\] in $V_{[\alpha]}$, and so $r \in \text{ann}_{L_K(R_n)}(c^{\infty})$, showing that $$L_K(R_n)(v - \phi_P(c))\subseteq \text{ann}_{L_K(R_n)}(c^{\infty}).$$ Hence $L_K(R_n)(v - \phi_P(c))= \text{ann}_{L_K(R_n)}(c^{\infty}).$ This implies that $$V^P_{[c^{\infty}]}\cong L_K(R_n)/L_K(R_n)(v - \phi_P(c)),$$ as desired. 	

(4) $(\Leftarrow)$ Assume that $d = \phi_P(\beta)$ for some $\beta\in \Pi_c$. Then, by \cite[Theorem 4.3]{anhnam} (see also \cite[Theorem 2.8]{amt:eosmolpa}), $V_{[d^{\infty}]} \cong L_K(R_n)/L_K(R_n)(v - d)$. Since $\beta\in \Pi_c$ and by Theorem \ref{Irrrep1}, $V_{[c^{\infty}]} \cong V_{[\beta^{\infty}]}$, and so $$V^P_{[c^{\infty}]} =(V_{[c^{\infty}]})^{\phi_{P^{-1}}} \cong (V_{[\beta^{\infty}]})^{\phi_{P^{-1}}} = V^P_{[\beta^{\infty}]}.$$
By Item (3), we have	$$V^P_{[\beta^{\infty}]}\cong L_K(R_n)/L_K(R_n)/L_K(R_n)(v - \phi_P(\beta))= L_K(R_n)/L_K(R_n)(v - d)\cong V_{[d^{\infty}]},$$ and so $V_{[d^{\infty}]}\cong V^P_{[c^{\infty}]}$, as desired.

$(\Rightarrow)$ Assume that $\theta: V_{[d^{\infty}]}\longrightarrow V^P_{[c^{\infty}]}$ is an isomorphism of left $L_K(R_n)$-modules. Let $q\in [d^{\infty}]$ be an element such that $\theta(q) = \sum^m_{i=1}k_i\alpha_i$, where $m$ is minimal such that $k_i\in K\setminus\{0\}$ and all the $\alpha_i$ are pairwise distinct in $[c^{\infty}]$. By repeating the method done in the proof of the direction $(\Rightarrow)$ of Theorem \ref{Irrrep2-irr} (4),  we obtain that $\tau_{\leq l}(\alpha_i) = \tau_{\leq l}(\alpha_j)$ for all $l\ge 1$ and $1\le i, j\le m$. Since all the $\alpha_i$ are pairwise distinct, we must have $m=1$. Since $q\in [d^{\infty}]$, $\tau_{>l}(p) = d^{\infty}$ for some $l\ge 0$, and so $$\theta( d^{\infty}) = \theta(\tau_{\leq l}(q)^* q) =\tau_{\leq l}(q)^*\cdot\theta(q)=   k_1\phi_{P^{-1}}(\tau_{\leq l}(q)^*)\alpha_1 = k\alpha,$$ where $k\in K\setminus\{0\}$ and $\alpha = \tau_{> l}(\alpha_1)$. This implies that
\[k\alpha=\theta(d^{\infty}) = \theta(d^t d^{\infty}) =d^t\cdot\theta(d^{\infty})=   k\phi_{P^{-1}}(d^t)\alpha\] for all $t\ge 1$, and so $\alpha = \beta^{\infty}$ for some $\beta\in SCP(R_n)$ and $\phi_{P^{-1}}(d) = \beta$. This shows that $d = \phi_P(\phi_{P^{-1}}(d)) = \phi_P(\beta)$. Since $\alpha\in [c^{\infty}]$, we have $[\beta^{\infty}]= [c^{\infty}]$, and so $\beta\in \Pi_c$,  as desired.

(5) We note that 
\begin{align*}
V^Q_{[d^{\infty}]}\cong V^P_{[c^{\infty}]}&\Longleftrightarrow  (V_{[d^{\infty}]})^{\phi_{P^{-1}}}\cong (V_{[c^{\infty}]})^{\phi_{Q^{-1}}}\Longleftrightarrow  (V_{[d^{\infty}]})^{\phi_{Q^{-1}}})^{\phi_Q}\cong (V_{[c^{\infty}]})^{\phi_{P^{-1}}})^{\phi_Q}\\
& \Longleftrightarrow V_{[d^{\infty}]}\cong (V_{[c^{\infty}]})^{\phi_{P^{-1}Q}} =V^{Q^{-1}P}_{[c^{\infty}]}.
\end{align*} 
Using this note and Item (4), we immediately get the statement, thus finishing the proof.
\end{proof}

In light of Theorem \ref{Irrrep5-irr}, we define an equivalence relation $\equiv$ on $SCP(R_n)\times GL_n(K)$ as follows: For all $(c, P)$ and $(d, Q)\in SCP(R_n)\times GL_n(K)$, $(c, P)\equiv (d, Q)$ if and only if $\phi_Q(d) = \phi_P(\beta)$ for some $\beta\in \Pi_c$. We denote by $[(c, P)]$ the $\equiv$-equivalence class of $(c, P)$. 
As a corollary of Theorem \ref{Irrrep5-irr}, we obtain that all  simple $L_K(R_n)$-modules $V^P_{[c^{\infty}]}$ associated to pairs $(\alpha, P)\in SCP(R_n)\times GL_n(K)$ may be parameterized by the set $(SCP(R_n)\times GL_n(K))/\equiv$.

\begin{cor}\label{Irrrep6-irr}
Let $K$ be a field, $n\ge 2$ a positive integer and $R_n$ the rose graph with $n$ petals. Then, the set $$\{V^P_{[c^{\infty}]}\mid [(c, P)]\in (SCP(R_n)\times GL_n(K))/\equiv\}$$ consists of pairwise non-isomorphic simple left $L_K(R_n)$-modules.	
\end{cor}
\begin{proof}
It immediately follows from Theorem \ref{Irrrep5-irr} (5).	
\end{proof}

We should note that for all $(P, Q)\in GL_n(K)\times GL_n(K)$ and $(c, d)\in SCP(R_n)\times SCP(R_n)$ with $|c|\neq |d|$, we always have $\deg(\phi_Q(d)) = |d|\neq |c| = \deg(\phi_P(\beta))$ for all $\beta\in \Pi_c$, and so $\phi_Q(d) \neq \phi_P(\beta)$ for all $\beta\in \Pi_c$. Consequently, we obtain that $[(c, P)]\neq [(d, Q)]$. This shows that there are infinitely many isomorphic classes of simple modules described in Corollary \ref{Irrrep6-irr}.  

Using Theorems \ref{Irrrep1}, \ref{Irrrep2-irr} and \ref{Irrrep5-irr}, we obtain a list of some classes of pairwise non-isomorphic simple modules for the Leavitt path algebra $L_K(R_n)$.
 
\begin{thm}\label{Irrrep7-irr}
Let $K$ be a field, $n\ge 2$ a positive integer and $R_n$ be the rose graph with $n$ petals. Then, the following simple left $L_K(R_n)$-modules
\begin{itemize} 
\item[(1)]  $V_{[\alpha]}$, where $\alpha\in (R_n)^{\infty}_{irr}$;

\item[(2)]  $S^f_{\Pi_c}$, where $c\in SCP(R_n)$ and $f\in \textnormal{Irr}(K[x])$;

\item[(3)]  $S^{f, p}_d$, where $d\in C_s(R_n)$, $f\in \textnormal{Irr}(K[x])$ with $\deg(f)\ge 2$, $[0]\neq [p]\in A_{R_n}(e_1, e_2)/\equiv_{f, d}$;

\item[(4)]  $V^P_{[\alpha]}$, where $[\alpha]_{\equiv} \in (R_n)^{\infty}_{irr-eeri}/\equiv \text{ and } I_n \neq P \in GL_n(K)$;

\item[(5)]  $V^P_{[c^{\infty}]}$, where  $[(c, P)]\in (SCP(R_n)\times GL_n(K))/\equiv$ and $P\neq I_n$
\end{itemize} 
are pairwise non-isomorphic.
\end{thm} 
\begin{proof}
By Theorem \ref{Irrrep1}, all the simple  modules $V_{[\alpha]}$, $S^f_{\Pi_c}$ and $S^{f, p}_d$ are  pairwise non-isomorphic. By Corollary \ref{Irrrep3-irr}, all $V^P_{[\alpha]}$ $([\alpha]_{\equiv} \in (R_n)^{\infty}_{irr-eeri}/\equiv \text{ and } P \in GL_n(K))$ are pairwise non-isomorphic. By Corollary \ref{Irrrep6-irr}, all $V^P_{[c^{\infty}]}$ $([(c, P)]\in (SCP(R_n)\times GL_n(K))/\equiv)$  are pairwise non-isomorphic. By Theorem \ref{Irrrep2-irr} (3), $V^P_{[\alpha]}$ is not finitely presented for all $\alpha\in (R_n)^{\infty}_{irr-eeri}$ and $P\in GL_n(K)$. While by Theorem \ref{Irrrep5-irr} (3),   $V^P_{[c^{\infty}]}$ is finitely presented for all $c\in SCP(R_n)$ and $P\in GL_n(K)$. By \cite[Theorem 3.6 (5)]{kn:ataanirolpa}, all $S^{f, p}_d$  are finitely presented. By \cite[Theorem 4.3]{anhnam} (see also \cite[Theorem 3.2]{kn:ataanirolpa}), all $S^f_{\Pi_c}$  are finitely presented. Therefore, each $V^P_{[\alpha]}$ is neither isomorphic to any $S^f_{\Pi_c}$ nor any $V^P_{[c^{\infty}]}$. By Theorem \ref{Irrrep5-irr} (2), $End_{L_K(R_n)}(V^P_{[c^{\infty}]})\cong K$ for all $c\in SCP(R_n)$ and $P\in GL_n(K)$. While by \cite[Theorem 3.6 (4)]{kn:ataanirolpa},  $End_{L_K(R_n)}(S^{f, p}_d)\cong K[x]/K[x]f(x)$ for all $d\in C_s(R_n)$, $f\in \text{Irr}(K[x])$ and $p\in A_{R_n}(e_1, e_2)$. Therefore, each $V^P_{[c^{\infty}]}$ is not isomorphic to any $S^{f, p}_d$ with $\deg(f)\ge 2$, thus finishing the proof.
\end{proof}

We end this article by presenting the following example which illustrates Theorem \ref{Irrrep7-irr}. 

\begin{exas}
Let $\mathbb{R}$ be the field of real numbers and $R_2$ be the rose with $2$ petals.  We then have $(R_2)^{\infty}_{irr-eeri} = (R_2)^{\infty}_{irr}$, and
$C_s(R_2) = \{e^m_1 e_2\mid m\in \mathbb{Z},\, m\ge 0\}$ and $A_{R_2}(e_1, e_2)$ is the $\mathbb{R}$-subalgebra of $L_{\mathbb{R}}(R_2)$ generated by $v, \, e_1, \, e^*_2$, that means, $$A_{R_2}(e_1, e_2) = \{\sum^n_{i =1}r_i e^{m_i}_1 (e^*_2)^{l_i}\mid n\ge 1,\, r_i \in \mathbb{R},\, m_i, l_i\ge 0 \},$$ where $e^0_1 = v = (e^*_2)^0$, and $\mathbb{R}[e_1]\subseteq A_{R_2}(e_1, e_2)$. By 
Corollary \ref{Irrrep7-irr}, the following simple left $L_{\mathbb{R}}(R_2)$-modules
\begin{itemize} 
\item[(1)]  $V_{[\alpha]}$, where $\alpha\in (R_2)^{\infty}_{irr}$;
	
\item[(2)]  $S^f_{\Pi_c}$, where $c\in SCP(R_2)$ and $f\in \text{Irr}(\mathbb{R}[x])$;
	
\item[(3)]  $S^{f, p}_{e^m_1e_2}$, where $m\ge 0$, $f = 1 - bx -ax^2\in \mathbb{R}[x]$ with $a\neq 0$ and $b^2 +4a<0$, and $0\neq p\in \mathbb{R}[e_1]$;
	
\item[(4)]  $V^P_{[\alpha]}$, where $[\alpha]_{\equiv} \in (R_2)^{\infty}_{irr}/\equiv \text{ and } I_2 \neq P \in GL_2(\mathbb{R})$;
	
\item[(5)]  $V^P_{[c^{\infty}]}$, where  $[(c, P)]\in (SCP(R_2)\times GL_2(\mathbb{R}))/\equiv$ and $P\neq I_2$
\end{itemize} 
are pairwise non-isomorphic.

\end{exas}

\vskip 0.5 cm \vskip 0.5cm {
	
\end{document}